\documentclass[final,onefignum,onetabnum]{siamart220329}

\usepackage{braket,amsfonts,amsmath,amssymb,mathrsfs,cases,mathdots,colordvi,url,color,algorithm,algorithmic,booktabs,footnote}
\usepackage{hyperref}
\usepackage{array}

\usepackage[caption=false]{subfig}
\usepackage{pgfplots}

\newsiamthm{claim}{Claim}
\newsiamremark{hypothesis}{Hypothesis}
\crefname{hypothesis}{Hypothesis}{Hypotheses}
\newsiamthm{asp}{Condition}

\newsiamthm{remark}{Remark}

\usepackage{algorithmic}

\usepackage{graphicx,epstopdf}

\Crefname{ALC@unique}{Line}{Lines}

\usepackage{amsopn}

\usepackage{xspace}
\usepackage{bold-extra}
\usepackage[most]{tcolorbox}

\colorlet{texcscolor}{blue!50!black}
\colorlet{texemcolor}{red!70!black}
\colorlet{texpreamble}{red!70!black}
\colorlet{codebackground}{black!25!white!25}

\lstdefinestyle{siamlatex}{%
  style=tcblatex,
  texcsstyle=*\color{texcscolor},
  texcsstyle=[2]\color{texemcolor},
  keywordstyle=[2]\color{texemcolor},
  moretexcs={cref,Cref,maketitle,mathcal,text,headers,email,url},
}

\tcbset{%
  colframe=black!75!white!75,
  coltitle=white,
  colback=codebackground, % bottom/left side
  colbacklower=white, % top/right side
  fonttitle=\bfseries,
  arc=0pt,outer arc=0pt,
  top=1pt,bottom=1pt,left=1mm,right=1mm,middle=1mm,boxsep=1mm,
  leftrule=0.3mm,rightrule=0.3mm,toprule=0.3mm,bottomrule=0.3mm,
  listing options={style=siamlatex}
}

\newtcblisting[use counter=example]{example}[2][]{%
  title={Example~\thetcbcounter: #2},#1}

\newtcbinputlisting[use counter=example]{\examplefile}[3][]{%
  title={Example~\thetcbcounter: #2},listing file={#3},#1}

\DeclareTotalTCBox{\code}{ v O{} }
{
  fontupper=\ttfamily\color{black},
  nobeforeafter,
  tcbox raise base,
  colback=codebackground,colframe=white,
  top=0pt,bottom=0pt,left=0mm,right=0mm,
  leftrule=0pt,rightrule=0pt,toprule=0mm,bottomrule=0mm,
  boxsep=0.5mm,
  #2}{#1}

\patchcmd\newpage{\vfil}{}{}{}
\begin{tcbverbatimwrite}{tmp_\jobname_header.tex}
\title{Strong Variational Sufficiency for Nonlinear Semidefinite Programming and its Implications\thanks{This version: May 4, 2023
}}

\author{Shiwei Wang\thanks{School of Mathematical Sciences, University of Chinese Academy of Science, Beijing, P.R. China. Institute of Applied Mathematics, Academy of Mathematics and Systems Science, Chinese Academy of Sciences, Beijing, P.R. China. (\email{wangshiwei182@mails.ucas.ac.cn}).}
\and Chao Ding\thanks{Institute of Applied Mathematics, Academy of Mathematics and Systems Science, Chinese Academy of Sciences, Beijing, P.R. China. School of Mathematical Sciences, University of Chinese Academy of Science, Beijing, P.R. China. (\email{dingchao@amss.ac.cn}). The work of this author was supported  in  part  by the Beijing Natural Science Foundation (Z190002), National Natural Science Foundation of China under project No. 12071464 and CAS Project for Young Scientists in Basic Research No. YSBR-034.}
\and Yangjing Zhang\thanks{Institute of Applied Mathematics, Academy of Mathematics and Systems Science, Chinese Academy of Sciences, Beijing, P.R. China. (\email{yangjing.zhang@amss.ac.cn}).The work of this author was supported by  the National Natural Science Foundation of China under project No. 12201617.}
\and Xinyuan Zhao\thanks{Department of Mathematics, Beijing University of Technology, Beijing, P.R. China. (\email{xyzhao@bjut.edu.cn}). The work of this author was supported  in  part  by  the National Natural Science Foundation of China under project No. 12271015 and No. 11871002. }}

\headers{ }{Shiwei Wang, Chao Ding, Yangjing Zhang and Xinyuan Zhao}
\end{tcbverbatimwrite}
\input{tmp_\jobname_header.tex}

\ifpdf
\hypersetup{ pdftitle={Strong Variational Sufficiency for Nonlinear Semidefinite Programming and its Implications} }
\fi

\begin{document}
\maketitle

\begin{tcbverbatimwrite}{tmp_\jobname_abstract.tex}
\begin{abstract}
Strong variational sufficiency is a newly proposed property, which turns out to be of great use in the convergence analysis of multiplier methods. However, what this property implies for non-polyhedral problems remains a puzzle. In this paper, we prove the equivalence between the strong variational sufficiency and the strong second order sufficient condition (SOSC) for nonlinear semidefinite programming (NLSDP), without requiring the uniqueness of multiplier or any other constraint qualifications. Based on this characterization, the local convergence property of the augmented Lagrangian method (ALM) for NLSDP can be established under strong SOSC in the absence of constraint qualifications. Moreover, under the strong SOSC, we can apply the semi-smooth Newton method to solve the ALM subproblems of NLSDP as the positive definiteness of the generalized Hessian of augmented Lagrangian function is satisfied.
\end{abstract}


\end{tcbverbatimwrite}
\input{tmp_\jobname_abstract.tex}

\section{Introduction}\label{sec:into}
The local optimality of the general optimization problem is a crucial topic for its theoretical importance and wide application. Traditionally, it is studied through the growth condition, e.g., the well-known first or second order optimality condition (cf. e.g., \cite{BShapiro00}). Recently, Rockafellar \cite{rock247} proposed a new property named strong variational sufficiency to deal with this topic geometrically. The key idea of this abstract definition originates from a so-called (strong) variational convexity, which indicates that the values and subgradients of a function are locally indistinguishable from those of a convex function. These two approaches seem to originate from different angles to understand the local optimality, but whether they possess deep connections is an essential issue as it provides not only a better comprehension of optimization theory, but also a solid theoretical foundation in algorithm design. Thus an explicit characterization of strong variational sufficiency is demanding.

 The definition of (strong) variational sufficient condition (Definition \ref{def:svasc}) is officially given in \cite{Roh22}  for the following general composite optimization problem
 \begin{equation}\label{eq:genp}
 	\min_{x\in \mathbb{X}}\quad f(x)+\theta(G(x)),
 \end{equation}
 where $\mathbb{X}$ and $\mathbb{Y}$ are two given Euclidean spaces, $f:\mathbb{X}\to \mathbb{R}$ and $G:\mathbb{X}\to \mathbb{Y}$ are twice continuously differentiable, and $\theta:\mathbb{Y}\to (-\infty,\infty]$ is a closed proper convex function. The strong variational sufficient condition is firstly introduced to deal with the local convexity of primal augmented Lagrangian function and the augmented tilt stability. It has gained more and more attention for its mathematical elegance and wide applications in the convergence analysis of multiplier methods. Several characterizations of this abstract property are also given in \cite{Roh22}. For instance, it is equivalent to the positive definiteness of the Hessian bundle of the augmented Lagrangian function \cite[Theorem 3]{Roh22} or the criterion \cite[Theorem 5]{Roh22} involving quadratic bundle (Definition \ref{def:quadratic bundle}).

 In terms of the connection with traditional optimality conditions, Rockafellar \cite[Theorem 4]{Roh22} shows that the strong variational sufficiency is equivalent to the well-known strong second order sufficient condition (SOSC), which is expressed entirely via the program data, when the function $\theta$ in \eqref{eq:genp} is polyhedral convex (i.e., the epigraph ${\rm epi}\,\theta$ of $\theta$ is a polyhedral convex set). For non-polyhedral problems,  the equivalence is still valid if $\theta$ in \eqref{eq:genp} is the indicator function of the second order cone with $G(\bar{x})\neq 0$, where $\bar{x}$ is a local optimal solution (see \cite[Example 3]{Roh22} for details). However, an explicit and verifiable characterization of strong variational sufficient condition for general non-polyhedral problems remains unknown. In this paper, without loss of generality, we mainly focus on the characterization of strong variational sufficiency for the following nonlinear semidefinite programming (NLSDP):
\begin{equation}\label{eq:NLSDPc}
	\begin{array}{cl}
		\displaystyle{\min_{x\in \mathbb{X}}} & f(x) \\ [3pt]
		{\rm s.t.} &{G(x)\in \mathbb{S}^n_+},
	\end{array}
\end{equation}
where  $\mathbb{S}^n$ is the linear space of all $n\times n$ real symmetric matrices equipped with the usual Frobenius inner product and its induced norm, $\mathbb{S}_+^n$ ($\mathbb{S}_-^n$) is the closed convex cone of all $n\times n$ positive (negative) semidefinite matrices in $\mathbb{S}^n$. One of our contributions lies in uncovering the equivalence between strong variational sufficient condition and strong SOSC (see Definition \ref{def:sigt} for details) for  NLSDP problems without requiring any other constraint qualifications.

An important application of the strong variational sufficiency of NLSDP is the local convergence analysis of the augmented Lagrangian method (ALM). The ALM, which was firstly proposed by Hestenes and Powell in 1969, and it has attained fruitful achievements during the past fifty years. Especially, \cite{rockalmppa1976} uncovers the equivalence between the proximal point algorithm (PPA) and ALM for convex problems. This work is of great importance in establishing the Q-linear convergence rate of the dual sequence for ALM. The conditions for the local linear convergence are further weakened (see e.g., \cite{Luque, CSToh2016}). For nonconvex case, the properties built up under the convexity assumption have become inadequate. Many efforts have been made to explore the convergence properties of ALM for nonconvex problems. See \cite{Bertsekas, Conn91, Lcb93, IKunisch90, FSolodov12, HSarabi20} for more details.
It is worth noting that in \cite{Bertsekas}, the author revealed a kind of local duality based on sufficient conditions for local optimality, which turns out to be the key to understanding the convergence of ALM for nonconvex nonlinear programming.
It follows that there may be a local reduction from nonconvex optimization to convex optimization.
This gives a hint on extending this local duality approach to broader nonconvex problems.
Recently, the newly published work \cite{Rockalm} opens new doors for the convergence rate analysis of ALM for nonconvex problems by using the aforementioned characterization of strong variational sufficient condition without any constraint qualifications.

For NLSDP \eqref{eq:NLSDPc}, the Lagrangian function is defined by
\begin{equation}\label{eq:L-function}
	L(x,Y):=f(x) +\langle Y,G(x)\rangle, \quad (x,Y)\in\mathbb{X}\times\mathbb{S}^n.
\end{equation}
For any $Y\in\mathbb{S}^n$, denote the first-order and second-order derivatives of $L(\cdot,Y)$ at $x\in\mathbb{X}$ by $L_x'(x,Y)$ and $L_{xx}''(x,Y)$, respectively. Given $\rho > 0$, the augmented Lagrangian function of \eqref{eq:NLSDPc} takes the following form (cf. \cite[Section 11.K]{RoWe98} and \cite{ShapS04})
\begin{equation}\label{eq:deflag}
	\mathcal{L}_{\rho}(x,Y):=f(x)+\frac{\rho}{2}{\rm dist}^2(G(x)+\frac{Y}{\rho},\mathbb{S}_+^n)-\frac{\|Y\|^2}{2\rho},
\end{equation}
where ${\rm dist}(z,\mathbb{S}_+^n)$ is the distance from point $z$ to $\mathbb{S}_+^n$.
For a given initial point $(x^0,Y^0)\in\mathbb{X}\times\mathbb{S}^n$ and a constant $\rho^0>0$, the $(k+1)$-th iteration of (extended) augmented Lagrangian method (ALM) for NLSDP \eqref{eq:NLSDPc} proposed by \cite{Rockalm} takes the following form
\begin{equation} \label{eq:subp1}
	\left\{\begin{array}{ll}
		x^{k+1}\approx\arg\min\{\mathcal{L}_{\rho^k}(x,Y^k)\}, \\ [3pt]
		Y^{k+1}=Y^k+\widetilde{\rho}^k\big[G(x^{k+1})-\Pi_{\mathbb{S}_+^n}(G(x^{k+1})+\frac{Y^k}{\rho^k})\big],
	\end{array}\right.
\end{equation}
where $\rho^{k}$, $\widetilde{\rho}^{k}>0$ and $\Pi_{\mathbb{S}_+^n}(\cdot)$ is the metric projection onto $\mathbb{S}_+^n$.  The (extended) ALM \eqref{eq:subp1}  reduces to the traditional ALM when $\widetilde{\rho}^{k}=\rho^{k}$.
To see how ALM \eqref{eq:subp1} works for nonconvex problems without constraint qualifications, consider the following simple example:
{\small\begin{equation}\label{toyexample}
	\begin{array}{cl}
		\min & \displaystyle{\frac{1}{2}x^3}\\
		{\rm s.t.} & -x^2\left[\begin{array}{ccc}
			0 & 0 & 0  \\
			0 & 0 & 0  \\
			0 & 0 & 1
		\end{array}\right]\in\mathbb{S}^3_+.
	\end{array}
\end{equation}}
The optimal solution of \eqref{toyexample} is $\bar{x}=0$ with the corresponding multiplier set ${\cal M}(\bar{x}):=\{Y\mid Y\in\mathbb{S}_-^3\}$. Due to the unboundedness of ${\cal M}(\bar{x})$, we know from \cite[Theorem 4.1]{ZKurcyusz79} that the Robinson constraint qualification \cite{Robinson76} does not hold at $\bar{x}$.  Pick a particular multiplier
$$\overline{Y}=\left[\begin{array}{ccc}
	0 & 0 & 0  \\
	0 & -1 & 0  \\
	0 & 0 & -2
\end{array}\right]\in {\cal M}(\bar{x}).$$ It is clear that $L''_{xx}(\bar{x},\overline{Y})=4>0$, which implies strong SOSC (Definition \ref{def:sigt}) holds at $(\bar{x},\overline{Y})$. We directly apply the  ALM (Algorithm \ref{algo1}) to problem \eqref{toyexample}, and we find that the corresponding ALM subproblem in \eqref{eq:subp1} can be solved exactly. Then, it can be observed from Figure \ref{fig1} that for fixed $\rho^k$ and $\widetilde{\rho}^k$, ${\rm dist}(Y^k,{\cal M}(\bar{x}))$, the distance between the $k$-th iteration $Y^k$ and ${\cal M}(\bar{x})$, converges to zero linearly.
\begin{figure}[h]
	\centering
	\includegraphics[scale=0.16]{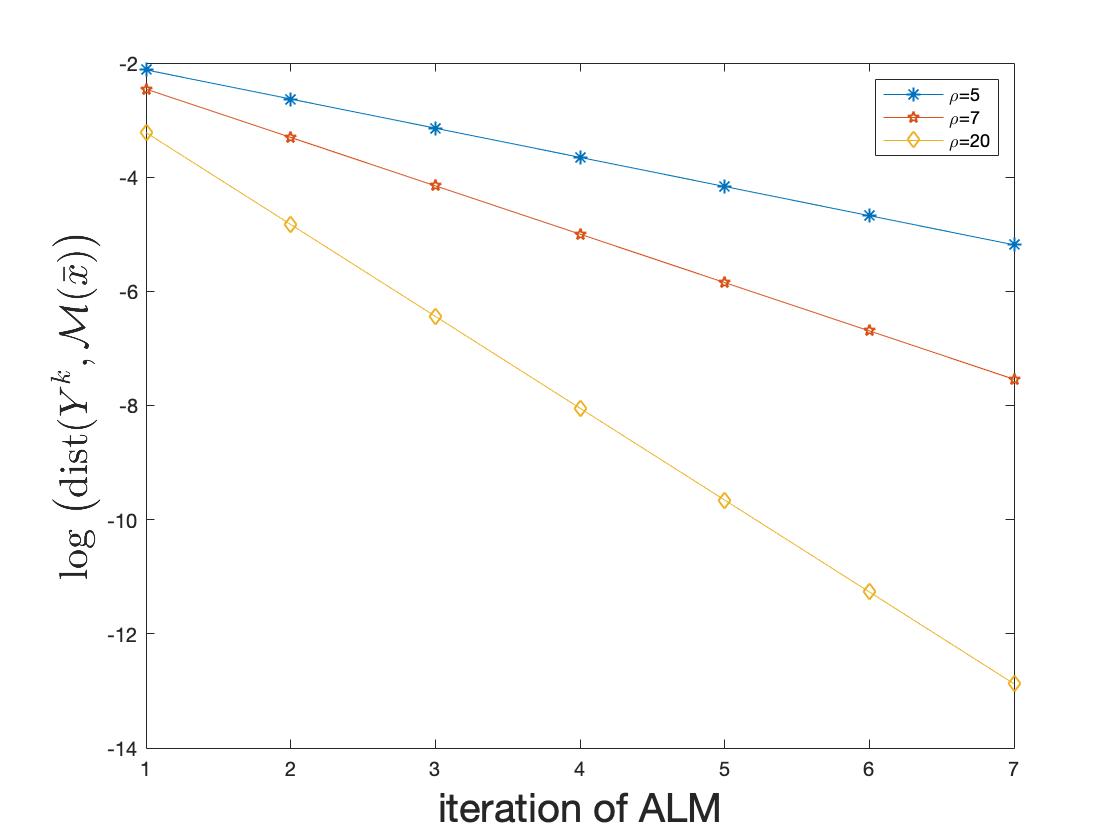}\caption{ALM for solving \eqref{toyexample} with different penalty parameters $\rho^k=\rho$.}\label{fig1}
\end{figure}

In this paper, combing \cite{Roh22} with our equivalence result, we obtain
the Q-linear convergence of dual variables and the R-linear convergence of primal ones of ALM for NLSDP under merely strong SOSC, without requiring the uniqueness of multiplier or any other constraint qualifications. It is worth noting that \cite{SSZhang08} and \cite{KSteck19} obtained the ALM local convergence result under (strong) SOSC and nondegeneracy \cite[Definition 3.2]{Sun06} or strict Robinson constraint qualification (SRCQ) (see e.g., \cite[(11)]{DSZhang} for its definition) respectively, both of which imply the uniqueness of multiplier.
Although \cite[Theorem 2]{WDing21} relaxed the uniqueness assumption of multiplier, additional assumptions \cite[Assumption 1]{WDing21} with respect to multipliers are still needed.

Another important and practical issue in applying ALM is how to solve the ALM subproblem \eqref{eq:subp1}. For convex problems, the successes of SDPNAL \cite{ZSToh}, SDPNAL$+$ \cite{STYZhao, YSToh} and QSDPNAL \cite{LSToh} have proved the efficiency of the semi-smooth Newton-CG method.  Its efficiency relies critically on the positive definiteness of certain generalized Hessian of the augmented Lagrangian function. And we are able to show the positive definiteness of the generalized Hessian under strong SOSC, as a direct application of \cite[Theorem 3]{Roh22} and our equivalence result.

The remaining parts of this paper are organized as follows. In the next section, we introduce some preliminary knowledge in semidefinite cone and strong variational sufficiency. In Section \ref{sec:main}, we propose our main equivalence result for both nonlinear second order cone programming (NLSOC) and NLSDP. Section \ref{sec:app}, as a direct application, copes with the convergence properties of ALM for NLSDP. Moreover, how to solve the subproblems and the corresponding convergence analysis are discussed. In Section \ref{sec:num}, the numerical experiment of an example is given to show the validity of the aforementioned results.  We conclude our paper and make some comments in the final section.

\section{Preliminaries}
For a given Euclidean space $\mathbb{X}$,
let $C$ be any subset in $\mathbb{X}$, ${\rm aff}\,{C}$ is the affine hull of $C$ \cite[page 6]{Rockafellar1970}. The Bouligand tangent/contingent cone of $C$ at $x$ is a closed cone defined by
\[
\mathcal{T}_{C}(x):= \left\{d\in\mathbb{X}\mid \mbox{$\exists\, t^k\downarrow 0$ and  $d^k\to d$ with  $x+t^kd^k\in C$ for all $k$} \right\}.
\]
When $C$ is convex, the normal cone in the sense of  convex analysis \cite{Rockafellar1970} is defined as
$$\mathcal{N}_{C}(x)=\left \{d \in \mathbb{X}\mid\langle d,x'-x\rangle\leq 0 \ \forall\,x'\in C \right\}.$$
For any $x\in C$, the critical cone  associated with ${ y}\in\mathcal{N}_C(x)$ of $C$ is defined in \cite[Page 98]{DoRo14}, i.e., ${\cal C}_{C}({x},{y})=\mathcal{T}_{C}({x})\cap({y})^{\perp}$.

For a set-valued mapping $\mathcal{F}:\mathbb{X}\rightrightarrows\mathbb{X}$, its $\limsup$ means
$$\limsup\limits_{{x}\rightarrow\bar{x}}\mathcal{F}({x}):=\left\{d\in\mathbb{X}\mid \exists\;x_k\rightarrow\bar{x},\; {y}_k\rightarrow d\;{\rm such\;that}\;{y}_k\in\mathcal{F}({x}_k)\;\forall\,k\right\}.$$
The following definitions of proximal and (general/Mordukhovich limiting) subdifferentials of functions are adopted from \cite[Definition 8.45 and 8.3]{RoWe98}.
\begin{definition}
	Consider a function $f:\mathbb{X}\rightarrow(-\infty, +\infty]$ and a point $\bar{x}$ with $f(\bar{x})$ finite. The proximal subdifferential of $f$ at $\bar{x}$ is defined as
	$$\partial^{\pi}f(\bar{x}):=\{{d}\mid \exists\; l>0, r>0\;{\rm such\; that}\;f(x)\geq f(\bar{x})+\langle{d},{x}-\bar{x}\rangle-l\|{x}-\bar{x}\|^2\;\forall\,{x}\in{\cal B}_{r}(\bar{x})\},$$
	where ${\cal B}_{r}(\bar{x})$ is the ball centered at $\bar{x}$ with radius $r$.
	The (general/Mordukhovich limiting) subdifferential of $f$ at $\bar{x}$ is $\partial f(\bar{x}):=\limsup\limits_{{x}\stackrel{f}{\rightarrow}\bar{x}}\partial^{\pi}f({x})$, where ${x}\stackrel{f}{\rightarrow}\bar{x}$ signifies ${x}\rightarrow\bar{x}$ with $f({x})\rightarrow f(\bar{x})$.
\end{definition}
Note that in an Euclidean space, the (general) subdifferential also can be constructed via the regular/Fr\'{e}chet subdifferentials  of functions $\widehat{\partial}f$ \cite[Definition 8.3]{RoWe98}. It is well-known that these two definitions are equivalent (see e.g., \cite[page 345]{RoWe98} or \cite[Theorem 1.89]{Mordukhovich06} for details). In particular, when $f$ is convex and finite at the corresponding point, the proximal and (general) subdifferentials coincide with the subdifferential in the sense of convex analysis \cite{Rockafellar1970}.

\subsection{Variational analysis of NLSDP}\label{sec:sub2.1}
Let  $A \in\mathbb{S}^n$ be given. We use $\lambda_{1}(A)\ge \lambda_2(A) \ge \ldots \ge \lambda_{n}(A)$ to denote the eigenvalues of $A$ (all real and counting multiplicity) arranging in  nonincreasing order and use $\lambda(A)$ to denote the vector of the ordered eigenvalues of $A$. Let $\Lambda(A):= {\rm Diag}(\lambda(A))$. Also, we use $v_1(A)> \dots> v_d(A)$ to denote the different eigenvalues and
$\zeta_l:=\{i\mid \lambda_i(A)=v_l(A)\}$.
Consider the eigenvalue decomposition of $A$, i.e., $A={P} \Lambda(A){P}^{T}$, where ${P}\in{\cal O}^{n}(A)$ is a corresponding
orthogonal matrix of the orthonormal eigenvectors.
By considering the index sets of positive, zero, and negative eigenvalues of $A$, we are able to write $A$ in the following form
{\footnotesize
\begin{equation}\label{eq:eig-decomp}
	A= \left[\begin{array}{ccc}
		{P}_{\alpha} & {P}_{\beta} & {P}_{\gamma}
	\end{array}\right] \left[\begin{array}{ccc}
		\Lambda(A)_{\alpha\alpha} & 0 & 0 \\ [3pt]
		0 & 0 & 0\\ [3pt]
		0 & 0 & \Lambda(A)_{\gamma\gamma}
	\end{array}\right]\left[\begin{array}{c}
		{P}_{\alpha}^T \\ [3pt]
		{P}_{\beta}^T \\
		{P}_{\gamma}^T
	\end{array}\right],
\end{equation}}
where $\alpha:=\{i: \lambda_i(A)>0\}$, $\beta:=\{i: \lambda_i(A)=0\}$ and $\gamma:=\{i: \lambda_i(A)<0\}$.

The critical cone of $\mathbb{S}_+^n$ at ${Y}\in{\cal N}_{\mathbb{S}_+^n}({X})$ with $A=X+Y$ is given by (cf. \cite[equations (11) and (12)]{CSun08})
\begin{equation}\label{eq:cricsdp}
	{\cal C}_{\mathbb{S}_+^n}({X},{Y})=\{H\in\mathbb{S}^n\mid P^T_{\beta}HP_{\beta}\in\mathbb{S}^{|\beta|}_+,\;P^T_{\beta}HP_{\gamma}=0,\;P^T_{\gamma}HP_{\gamma}=0\}
\end{equation}
and the affine hull of ${\cal C}_{\mathbb{S}_+^n}({X},{Y})$ is
\begin{equation}\label{eq:affcsdp}
	{\rm aff}\,{\cal C}_{\mathbb{S}_+^n}({X},{Y})=\{H\in\mathbb{S}^n\mid P^T_{\beta}HP_{\gamma}=0,\;P^T_{\gamma}HP_{\gamma}=0\}.
\end{equation}
The  NLSDP \eqref{eq:NLSDPc} can be written in the following form
\begin{eqnarray}
	%\mbox{(SDP)}~~~~
	\min\quad && f(x)+\delta_{\mathbb{S}^n_+}(G(x)), \label{eq:NLSDP}
\end{eqnarray}
where
$\delta_{\mathbb{S}^n_+} : \mathbb{S}^n\rightarrow (-\infty,\infty]$ is the indicator function of positive semidefinite cone $\mathbb{S}^n_+$.
Denote ${S}_{KKT}(a,b)$ the solution set of the KKT optimality condition for problem \eqref{eq:NLSDP}, i.e.,
\begin{equation}\label{eq:KKT-NLSDP-p1}
	{S}_{KKT}(a,b)=\left\{(x,Y)\in \mathbb{X}\times\mathbb{S}^n \Bigg | \left.
	\begin{array}{l}
		L_x'(x,Y)-a=0,\\ [3pt]
		\mathbb{S}_+^n\ni (G(x)-b)\perp Y \in \mathbb{S}_-^n.
	\end{array}
	\right.\right \}.
\end{equation}
For any KKT pair $(\bar{x},\overline{Y})$ that satisfies the KKT condition with $(a,b) = (0,0)$, we call $\bar{x}$ a stationary point.
Suppose $\bar{x}$ is a stationary point. Define ${\cal M}(\bar{x})$ as the set of all multipliers $Y\in\mathbb{S}^n$ satisfying the KKT condition \eqref{eq:KKT-NLSDP-p1}, i.e.,
\begin{equation}\label{eq:defcalm}
	{\cal M}(\bar{x})=\{Y\in\mathbb{S}^n \mid (\bar{x},Y)\in {S}_{KKT}(0,0)\}.
\end{equation}

The definition of strong SOSC for NLSDP, which demands the supremum of \eqref{eq:defssosc} over $Y\in{\cal M}(\bar{x})$ holds, is originally given by \cite{Sun06} as an analogue for Robinson \cite{Robin80} for the nonlinear programming (NLP). However, the strong SOSC mentioned here is slightly different in only requiring the validity of \eqref{eq:defssosc} at $\overline{Y}$. It is worth to note when ${\cal M}(\bar{x})$ is a singleton, both of them are the same.
\begin{definition}\label{def:sigt}\cite[Definition 3.2]{Sun06}
	Let $\bar{x}$ be a stationary point of NLSDP \eqref{eq:NLSDP} and $\overline{Y}\in{\cal M}(\bar{x})$. We say the strong second order sufficient condition (SOSC) holds at $(\bar{x},\overline{Y})$ if
	\begin{equation}\label{eq:defssosc}
		\langle L_{xx}''(\bar{x},\overline{Y})d,d\rangle-\varUpsilon_{G(\bar{x})}\big(\overline{Y},G'(\bar{x}){d}\big)>0\quad\forall\; 0\neq G'(\bar{x})d\in{\rm aff}\,{\cal C}_{\mathbb{S}_+^n}(G(\bar{x}),\overline{Y}).
	\end{equation}
	where $\varUpsilon_{G(\bar{x})}\big(\overline{Y},G'(\bar{x}){d}\big)=2\langle\overline{Y},(G'(\bar{x}){d})G(\bar{x})^{\dagger}(G'(\bar{x}){d})\rangle$ is the $\sigma$-term and $G(\bar{x})^{\dagger}$ is the generalized inverse matrix of $G(\bar{x})$.
\end{definition}

Suppose $A=G(\bar{x})+\overline{Y}$ possesses the decomposition \eqref{eq:eig-decomp}. From \cite[page 386]{CSun08}, we know the $\sigma$-term takes the explicit form of
\begin{equation}\label{eq:sigtsdp}
	\varUpsilon_{G(\bar{x})}\big(\overline{Y},H\big)=2\sum_{i\in\alpha,j\in\gamma}\frac{\lambda_j(A)}{\lambda_i(A)}(\widetilde{H}_{ij})^2,
\end{equation}
where $\widetilde{H}=P^THP$.

Next, we introduce some essential notation for our main result. The (second order) generalized differentiability  of the augmented Lagrangian function ${\cal L}_{\rho}$ defined by \eqref{eq:deflag} for NLSDP \eqref{eq:NLSDP} has been explicitly established in \cite{SSZhang08}.  In fact, it is well-known that ${\cal L}_{\rho}$ is continuously differentiable with
\begin{equation}\label{eq:algrad}
	({\cal L}_{\rho})_x'(x,Y)=f'(x)+\rho G'(x)^*\Pi_{\mathbb{S}_-^n}(G(x)+\rho^{-1}Y),
\end{equation}
where $\Pi_{\mathbb{S}_-^n}(\cdot)$ is the metric projection over $\mathbb{S}_{-}^n$ and $G'(x)^*$ denotes the adjoint of the corresponding linear mapping. Moreover, since $({\cal L}_{\rho})_x'$ is  Lipschitz continuous, by using Rademacher’s theorem, the B(ouligand)-subdifferential of $({\cal L}_{\rho})_x'$ at $(x,Y)$ is given by
\begin{equation}\label{eq:partb}
	\partial_B\big(({\cal L}_{\rho})'_x\big)(x,Y):=\{\lim_{k\rightarrow\infty}(({\cal L}_{\rho})_{x}')'(x^k,Y^k)\mid (x^k,Y^k)\in{\cal U},\ (x^k,Y^k)\rightarrow(x,Y)\},
\end{equation}
where ${\cal U}$ is the set of Fr\'{e}chet-differentiable points of $({\cal L}_{\rho})_x'$.
It follows from \cite[(18)]{SSZhang08} (using \cite[Lemma 2 and 3]{SSZhang08}) that
\begin{align*}
	&\pi_x\partial_B\big(({\cal L}_{\rho})_x'\big)(x,Y)(\Delta x)\\
	&=L_{xx}''(x,\rho\Pi_{\mathbb{S}_-^n}(G(x)+\rho^{-1}Y))(\Delta x)+\rho G'({x})^*\partial_B\Pi_{\mathbb{S}_-^n}(G({x})+\rho^{-1}Y) G'({x})(\Delta x),
\end{align*}
where $\pi_x\partial_B\big(({\cal L}_{\rho})_x'\big)(x,Y)$ is the projection of $\partial_B\big(({\cal L}_{\rho})_x'\big)(x,Y)$ onto the space $\mathbb{X}$ and  $\partial_B\Pi_{\mathbb{S}_-^n}(G({x})+\rho^{-1}Y)$ is the B-subdifferential of $\Pi_{\mathbb{S}_-^n}(\cdot)$ at $G({x})+\rho^{-1}Y$. Actually, the set $\pi_x\partial_B\big(({\cal L}_{\rho})_x'\big)(x,Y)$ is recently defined as the $x$ part of the Hessian bundle \cite[(3.1)]{Roh22} for the augmented Lagrangian function ${\cal L}_{\rho}$ at $(x,Y)$ (see \cite[(3.6)]{Roh22} for detail).

Let $\bar{x}$ be the stationary point of \eqref{eq:NLSDP}. For each $\overline{Y}\in{\cal M}(\bar{x})$ and $W\in\partial_B\Pi_{\mathbb{S}_-^n}(G(\bar{x})+\rho^{-1}\overline{Y})$, define the following mapping
\begin{equation}\label{eq:algrad2}
	{\cal A}_{\rho}(\overline{Y},W):=L_{xx}''(\bar{x},\overline{Y})+\rho G'(\bar{x})^*WG'(\bar{x}).
\end{equation}
Let $\overline{Z}=G(\bar{x})+\overline{Y}$ possesses decomposition \eqref{eq:eig-decomp} with $\overline{P}\in{\cal O}^{n}(\overline{Z})$. It follows from \cite[Lemma 5]{SSZhang08} that $W\in\partial_B\Pi_{\mathbb{S}_-^n}(\overline{Z})$ if and only if there exists $W_0\in\partial_B\Pi_{\mathbb{S}_-^{\lvert\beta\rvert}}(0)$ such that for all $H\in\mathbb{S}^n$
{\footnotesize
\begin{equation*}
	W(H)=\overline{P}\left[\begin{array}{ccc}
		0 & 0 & \Sigma_{\alpha\gamma}\circ\overline{P}_{\alpha}^TH\overline{P}_{\gamma}\\
		0 & W_0(\overline{P}_{\beta}^TH\overline{P}_{\beta}) & \overline{P}_{\beta}^TH\overline{P}_{\gamma} \\
		\Sigma_{\gamma\alpha}\circ\overline{P}_{\gamma}^TH\overline{P}_{\alpha} & \overline{P}_{\gamma}^TH\overline{P}_{\beta} & \overline{P}_{\gamma}^TH\overline{P}_{\gamma}
	\end{array}\right]\overline{P}^T,
\end{equation*}}
where
\begin{equation*}
\begin{cases}		\Sigma_{ij}=1-\frac{\max\{\lambda_i,0\}+\max\{\lambda_j,0\}}{\lvert\lambda_i\rvert+\lvert\lambda_j\rvert}, & (i,j)\notin\beta\times\beta\\
\Sigma_{ij}	\in[0,1] & (i,j)\in\beta\times\beta
	\end{cases}
\end{equation*}
with $\lambda_i:=\lambda_i(\overline{Z})$ for short. Let ${\cal Q}^n$ denote the set of $n$-dimensional orthogonal matrix.
Also, $W_0\in\partial_B\Pi_{\mathbb{S}_-^{\lvert\beta\rvert}}(0)$ if and only if there exist $Q\in{\cal Q}^{\lvert\beta\rvert}$ and $\Omega\in\mathbb{S}^{\lvert\beta\rvert}$ with entries $\Omega_{ij}\in[0,1]$ such that for all $H\in\mathbb{S}^{\lvert\beta\rvert}$,
$$W_0(H)=Q(\Omega\circ(Q^THQ))Q^T.$$
By \cite[Lemma 9]{SSZhang08}, we have
{\small \begin{align}
	\langle d,{\cal A}_{\rho}(\overline{Y},W)d\rangle=&\langle d,L_{xx}''(\bar{x},\overline{Y})d\rangle+\rho\sum_{i,j\in\gamma}(\overline{P}^T(G'(\bar{x})d)\overline{P})_{ij}^2+2\rho\sum_{i\in\beta,j\in\gamma}(\overline{P}^T(G'(\bar{x})d)\overline{P})_{ij}^2\label{eq:expcalA}\\
	&+2\rho\sum_{i\in\alpha,j\in\gamma}\frac{-\lambda_j}{\rho\lambda_i-\lambda_j}(\overline{P}^T(G'(\bar{x})d)\overline{P})_{ij}^2+\rho\sum_{i,j\in\beta}(\Omega_{\rho})_{ij}({P}^T(G'(\bar{x})d){P})_{ij}^2,\nonumber
\end{align}}
where $\Omega_{\rho}\in\mathbb{S}^{\lvert\beta\rvert}$ with entries $(\Omega_{\rho})_{ij}\in[0,1]$, $P=[\overline{P}_{\alpha}\;\overline{P}_{\beta}Q\;\overline{P}_{\gamma}]$ with $Q\in{\cal Q}^{\lvert\beta\rvert}$.

\subsection{Strong variational sufficiency}
The definition of (strong) variational sufficient condition is given in \cite{Roh22}  for general composite optimization problem \eqref{eq:genp}.
We can recast \eqref{eq:genp} in the form
\begin{equation}\label{eq:pgpert}
	\min\;\phi(x,u)\;\mbox{subject to}\;u=0,\quad\mbox{where}\;\phi(x,u)=f(x)+\theta(G(x)+u).
\end{equation}
The first order local optimality condition for \eqref{eq:genp}  of $\bar{x}$ is the existence of $\overline{Y}$ such that
$$L_x'(\bar{x},\overline{Y})=0\quad \mbox{with}\quad\overline{Y}\in\partial\, \theta(G(\bar{x})). $$
Define
\begin{equation}\label{eq:phir}
\phi_r:=\phi(x,u)+\frac{r}{2}|u|^2.
\end{equation}
The variational (strong) convexity, which is firstly proposed in \cite{rock247}, refers to the existence of open convex neighborhoods ${\cal W}$ of $(\bar{x},0)$ and ${\cal Z}$ of $(0,\overline{Y})$ such that there exists a proper closed (strongly) convex function $\psi\leq\phi_r2$ on ${\cal W}$ such that
\begin{equation*}
	({\cal W}\times{\cal Z})\cap{\rm gph}\,\partial\psi=({\cal W}\times{\cal Z})\cap{\rm gph}\,\partial\phi_r
\end{equation*}
and for $(x,u;v,y)$ belonging to this common set, $\psi(x,u)=\phi_r(x,u)$.
\begin{definition}\label{def:svasc}\cite{Roh22}
	The (strong) variational sufficient condition for local optimality in \eqref{eq:pgpert} holds with respect to $\bar{x}$ and $\overline{Y}$ satisfying the first order condition if there exists $r>0$ such that $\displaystyle\phi_r(x,u)$ is variationally (strongly) convex with respect to the pair $\big((\bar{x},0),(0,\overline{Y})\big)$ in ${\rm gph}\,\partial\phi_r$.
\end{definition}

To characterize the above property, the tools of the second subderivative and generalized quadratic form are needed.
\begin{definition}\cite[Definition 13.6]{RoWe98}
	Let $\bar{x}$ be a point where the function $f:\mathbb{X}\rightarrow[-\infty,+\infty]$ is finite. $f$ is twice epi-differentiable at $\bar{x}$ for $v$ if the functions
	$$\Delta^2_{t}f(\bar{x}\mid v)(u)=\frac{f(\bar{x}+tu)-f(\bar{x})-t\langle v,u\rangle}{\frac{1}{2}t^2}$$
	epi-converge to $d^2f(\bar{x}\mid v)$ as $t\downarrow0$, where $d^2f(\bar{x}\mid v)$ is the second subderivative of $f$ at $\bar{x}$ for $v$ defined as
	$$d^2f(\bar{x}\mid v)(w)=\liminf_{t\downarrow0,u\rightarrow w}\Delta^2_{t}f(\bar{x}\mid v)(u)$$
\end{definition}
We know from \cite[Theorem 3.6]{mohammadi20} that $\delta_{\mathbb{S}_+^n}$ is twice epi-differentiable at $X$ for $Y$ with $Y\in{\cal N}_{\mathbb{S}_+^n}(X)$.

\begin{definition}\label{def:genguadf} \cite{Roh22}
	A generalized linear mapping ${\cal R}$ from $\mathbb{X}$ to $\mathbb{Y}$, is a set-valued mapping for which ${\rm gph}\,{\cal R}$ is a subspace of $\mathbb{X}\times\mathbb{Y}$. This means that ${\rm dom}{\cal R}$ is a subspace $\mathbb{Z}$ of $\mathbb{X}$, ${\cal R}(0)$ is a subspace $\mathbb{Z}'$ of $\mathbb{Y}$, and there is an ordinary linear mapping ${\cal R}_0:\mathbb{Z}\rightarrow\mathbb{Y}$ such that ${\cal R}(x)={\cal R}_0(x)+\mathbb{Z}'$ for $x\in\mathbb{Z}$.
	We call a function $q: \mathbb{X}\rightarrow(-\infty,+\infty]$ is a generalized quadratic form on $\mathbb{X}$ if $q(0)=0$ and the subgradient mapping $\partial q:\mathbb{X}\rightrightarrows\mathbb{X}$ is generalized linear. A function $g$ on $\mathbb{X}$ will be called generalized twice differentiable at $x$ for a subgradient $y$ if it is twice epi-differentiable at $x$ for $y$ with the second-order subderivative $d^2g(x\mid y)$ being a generalized quadratic form.
\end{definition}

The following definition of the quadratic bundle is taken from \cite{Roh22}, which is essential in characterizing the strong variational sufficiency.
\begin{definition}\label{def:quadratic bundle}
	For general optimization problem \eqref{eq:genp}, suppose $(\bar{x},\overline{Y})$ is a KKT pair. The quadratic bundle of $\theta$ is defined as
	\begin{equation*}
		{\rm quad}\,\theta(G(\bar{x})\mid\overline{Y})=\left\{
		\begin{array}{l}
			\mbox{the collection of generalized quadratic forms}\; q\; \mbox{for which}\\
			\exists(X^k,Y^k)\rightarrow(G(\bar{x}),\overline{Y})\;\mbox{with}\;\theta\;\mbox{generalized twice differentiable}\\
			\mbox{at}\;X^k\;\mbox{for}\;Y^k\;\mbox{and such that the generalized quadratic}\\
			\mbox{forms}\;q_k=\frac{1}{2}d^2\theta(X^k\mid Y^k)\;\mbox{converge epigraphically to}\;q.
		\end{array}
		\right.
	\end{equation*}
\end{definition}
As mentioned in \cite{Roh22}, the variational sufficient condition guarantees the local optimality for \eqref{eq:pgpert}. The following result is taken from \cite[Theorem 5]{Roh22}, which is useful for the subsequent analysis.
\begin{proposition}\label{prop:thm5}
	For general optimization problem \eqref{eq:genp},
	strong variational sufficient condition for local optimality with respect to $(\bar{x},\overline{Y})$ is equivalent to that  every $q\in{\rm quad}\,\theta(G(\bar{x})\mid\overline{Y})$ has
	\begin{equation}\label{svs-equ}
		\frac{1}{2}\langle L_{xx}''(\bar{x},\overline{Y})d,d\rangle+q(G'(\bar{x}){d})>0\quad{\rm when}\;d\neq0,
	\end{equation}
	where ${\rm quad}\,\theta(G(\bar{x})\mid\overline{Y})$ is the quadratic bundle of $\theta$.
\end{proposition}

\section{The characterization of strong variational sufficient condition for NLSDP}\label{sec:main}
In this section, we will study the relationship between strong variational sufficient condition and the well-known strong SOSC \eqref{eq:defssosc} for NLSDP by combining  \eqref{svs-equ} and \cite[Theorem 3.3]{mohammadi20} together. Firstly, we use nonlinear second order cone programming (NLSOC) as an example to illustrate our approach. The general NLSOC  can be written as
\begin{equation}\label{eq:SOC}
	\begin{array}{cl}
		\displaystyle{\min_{x\in \mathbb{X}}} & f(x) \\ [3pt]
		{\rm s.t.} &{g(x)\in {\cal K}},
	\end{array}
\end{equation}
where $f,g$ are both twice continuously differentiable, ${\cal K}$ is the second order cone defined as ${\cal K}=\{(u_1,\dots,u_m)\in\mathbb{R}^m\mid u_1\geq\|(u_2,\dots,u_m)\|\}=\{u\mid h(u)\leq0\} $
with $h(u)=-u_1+\|(u_2,\dots,u_m)\|$.
The ordinary Lagrangian function of problem \eqref{eq:SOC}
is defined by
\begin{equation}\label{eq:L-functionsoc}
	L(x,y):=f(x) +\langle y,g(x)\rangle, \quad (x,y)\in\mathbb{X}\times\mathbb{R}^m.
\end{equation}
Given a stationary point $\bar{x}$. Let
$${\cal M}_{\rm soc}(\bar{x}):=\left\{y\in \mathbb{R}^m\Bigg | \left.
\begin{array}{l}
	L_x'(\bar{x},y)=0,\\ [3pt]
	{\cal K}\ni g(\bar{x})\perp y \in {\cal K}^{\circ}
\end{array}
\right.\right \}. $$
be the set of all multipliers $y\in\mathbb{R}^m$ satisfying the KKT condition for \eqref{eq:SOC}, where ${\cal K}^{\circ}$ is the polar cone of ${\cal K}$ defined in \cite[Section 14]{Rockafellar1970}.
The strong SOSC at $(\bar{x},\bar{y})$ is defined as
\begin{equation}\label{eq:defssoscsoc}
	\langle L_{xx}''(\bar{x},\bar{y})d,d\rangle-\varUpsilon_{g(\bar{x})}\big(\bar{y},g'(\bar{x}){d}\big)>0 \quad\forall\; 0\neq g'(\bar{x})d\in{\rm aff}\,{\cal C}_{\cal K}(g(\bar{x}),\bar{y}),
\end{equation}
where the explicit form of $\varUpsilon_{g(\bar{x})}\big(\bar{y},g'(\bar{x}){d}\big)$ is given in \cite[Theorem 29]{bsoc}.

As illustrated in \cite[Example 3]{Roh22}, the explicit form of the quadratic bundle ${\rm quad}\,\delta_{\cal K}(g(\bar{x})\mid\bar{y})$ for certain reference KKT point $(\bar{x},\bar{y})$ can be obtained by using corresponding results for polyhedral problems. For example,
if $g(\bar{x})\in{\rm int}\,{\cal K}$, we have $\bar{y}=0$. It can be checked directly by using \cite[Theorem 3.3]{mohammadi20} and \cite[(36),(43)]{bsoc} that the quadratic bundle ${\rm quad}\,\delta_{\cal K}(g(\bar{x})\mid \bar{y})$ consists of solely $q\equiv0$.
If $g(\bar{x})\in{\rm bd}\,{\cal K}\backslash\{0\}$, for any nonzero $\bar{y}\in{\cal N}_{\cal K}(g(\bar{x}))$, the quadratic bundle ${\rm quad}\,\delta_{\cal K}(g(\bar{x})\mid \bar{y})$ consists of the generalized quadratic form
\begin{equation}\label{eq:qbsoc1}
	q=\frac{1}{2}d^2\delta_{\cal K}(g(\bar{x})\mid \bar{y})\quad{\rm with}\quad q(w)=\left\{\begin{array}{cc}
		\frac{1}{2}w\cdot h''(g(\bar{x}))w & {\rm if}\;h'(g(\bar{x}))w=0, \\
		\infty & {\rm otherwise. }
	\end{array} \right.
\end{equation}
If $g(\bar{x})=0$ and $\bar{y}\in{\rm int}(-{\cal K})$, the quadratic bundle consists of $q=\delta_{\{0\}}$. If $g(\bar{x})=0$ and $\bar{y}\in{\rm bd}(-{\cal K})\backslash\{0\}$, both $\delta_{\{0\}}$ and \eqref{eq:qbsoc1} constitute the quadratic bundle.  It can be checked directly that for NLSOC, strong variational sufficient condition is equivalent to strong SOSC when the reference point $(\bar{x},\bar{y})$ lies in one of the above circumstances.

Thus,  only two circumstances for NLSOC remain to be discussed.
The first one is $g(\bar{x})=0$ and $\bar{y}=0$. Since ${\cal K}$ is ${\cal C}^2-$cone reducible, we know from \cite[Theorem 3.3]{mohammadi20}, \cite[Theorem 29]{bsoc} that for any $(g^k,y^k)\rightarrow(g(\bar{x}),\bar{y})$ and $w\in\mathbb{R}^m$,
\begin{equation}\label{eq:qbgeq0}
d^2\delta_{{\cal K}}(g^k\mid y^k)(w)=w^T{\cal H}(g^k,y^k)w+\delta_{{\cal C}_{\cal K}(g^k,y^k)}(w)\geq0,
\end{equation}
where ${\cal H}(g^k,y^k)=\frac{y^k_1}{g^k_1}(g^k)'^T(1\;0^T;0\;-I_{m-1})(g^k)'$ if $g^k\in{\rm bd}\,{\cal K}\backslash\{0\}$ and ${\cal H}(g^k,y^k)=0$ otherwise. If we pick $g^k\in{\rm int}\,{\cal K}\rightarrow g(\bar{x})$ and $y^k=0$ for all $k$, it is easy to see from \cite[Theorem 29, (35), (36)]{bsoc} that
$$d^2\delta_{{\cal K}}(g^k\mid y^k)(w)=0+\delta_{{\cal C}_{\cal K}(g^k,y^k)}(w),$$
where ${\cal C}_{\cal K}(g^k,y^k)=\mathbb{R}^m$. Thus we have proved $0\in{\rm quad}\,\delta_{\cal K}(g(\bar{x})\mid \bar{y})$. Combining this with \eqref{eq:qbgeq0}, it follows that in this case, \eqref{svs-equ} is equivalent to strong SOSC.

The second case is $g(\bar{x})\in{\rm bd}{\cal K}\backslash\{0\}$ and $\bar{y}=0$. Let $g^k\in{\rm int}\,{\cal K}\rightarrow g(\bar{x})$ and $y^k=0$ for each $k$. We know that
$$d^2\delta_{{\cal K}}(g^k\mid y^k)(w)=0+\delta_{{\cal C}_{\cal K}(g^k,y^k)}(w),$$
where ${\cal C}_{\cal K}(g^k,y^k)=\mathbb{R}^m$. As $k\rightarrow\infty$, we know that $\lim_{k\rightarrow\infty}d^2\delta_{{\cal K}}(g^k\mid y^k)(w)=0$. Using \cite[(35), (36), (43)]{bsoc} again, we also have \eqref{svs-equ} is equivalent to strong SOSC.

Thus, we immediately obtain the following characterization of strong variational sufficiency for NLSOC, which is a supplement for \cite[Example 3]{Roh22}.
\begin{proposition}\label{thm:varscequissoscsoc}
	Let $\bar{x}\in\mathbb{X}$ be a stationary point to the NLSOC \eqref{eq:SOC} and $\bar{y}\in {\cal M}(\bar{x})$. The strong variational sufficient condition \eqref{svs-equ} with respect to $(\bar{x},\bar{y})$ holds if and only if  the strong SOSC \eqref{eq:defssoscsoc} holds at $(\bar{x},\bar{y})$.
\end{proposition}

It is worth noting that the above proposition is not surprising. If $g(\bar{x})\neq0$, we can obtain this result by regarding the NLSOC problem as a polyhedral problem as shown in \cite[Example 3]{Roh22}.
If $g(\bar{x})=0$, the $\sigma-$term happens to be 0, which makes the calculation of the quadratic bundle much simpler. When it comes to NLSDP problem \eqref{eq:NLSDP}, things are not so easy as we can neither regard SDP as a polyhedral problem nor have a simplification for its $\sigma-$term \eqref{eq:sigtsdp}. However, the success of this approach to NLSOC gives us a hint that this approach may also work for NLSDP.
Before we put forward our main result, we need the following proposition on the quadratic bundle defined by Definition \ref{def:quadratic bundle} for $\delta_{\mathbb{S}_+^n}$.
\begin{proposition}\label{prop:onesidessosc}
	Let $\bar{x}\in\mathbb{X}$ be a stationary point to NLSDP \eqref{eq:NLSDP} and $\overline{Y}\in {\cal M}(\bar{x})$, where ${\cal M}(\bar{x})$ is given in \eqref{eq:defcalm}. Let $A=G(\bar{x})+\overline{Y}$, which possesses the decomposition \eqref{eq:eig-decomp}. Then, there exists $q\in {\rm quad}\,\delta_{\mathbb{S}_+^n}(G(\bar{x})\mid\overline{Y})$ such that for all $H\in\mathbb{S}^n$,
	\begin{align}
		q(H)&=-\frac{1}{2}\varUpsilon_{G(\bar{x})}\big(\overline{Y},H\big)+\delta_{{\rm aff}\,{\cal C}_{\mathbb{S}_+^n}(G(\bar{x}),\overline{Y})}(H)\nonumber\\
		&=\sum_{i\in\alpha,j\in\gamma}\frac{-\lambda_j(A)}{\lambda_i(A)}(\widetilde{H}_{ij})^2+\delta_{{\rm aff}\,{\cal C}_{\mathbb{S}_+^n}(G(\bar{x}),\overline{Y})}(H),\label{eq:qeqsg}
	\end{align}
	where $\widetilde{H}=P^THP$.
\end{proposition}
\begin{proof}
	For each $k$, choose
	\[
	X^{k}={P} \left[ \begin{array}{ccccc} \Lambda(A)_{\alpha \alpha} & 0
		& 0  \\   0 & (z^{k})_{\beta} & 0 \\   0 & 0 &
		0 \\
	\end{array} \right] {P}^{T} \quad {\rm and} \quad Y^{k}=\overline{Y},
	\]
	where $z^k\downarrow0$ (this notation means for each $k$, $z^k>0$ and $z^k\rightarrow0$ as $k\rightarrow\infty$) with each $(z^k)_i$ non-increasing on $k$. Let $A^k=X^k+Y^k$ for each $k$. Since $\mathbb{S}_+^n$ is ${\cal C}^2-$cone reducible \cite[Example 3.140]{BShapiro00}, we know from \cite[Theorem 3.3]{mohammadi20}, \eqref{eq:cricsdp} and \eqref{eq:sigtsdp} that for any $H\in\mathbb{S}^n$,
	\begin{align*}
		\frac{1}{2}d^2\delta_{\mathbb{S}_+^n}(X^k\mid Y^k)(H)=&\sum_{i\in\alpha\cup\beta,j\in\gamma}\frac{-\lambda_j(A^k)}{\lambda_i(A^k)}\big(\widetilde{H}_{ij}\big)^2+\delta_{{\cal C}_{\mathbb{S}_+^n}(X^k,Y^k)}(H)\\
		=&\sum_{i\in\alpha,j\in\gamma}\frac{-\lambda_j(A)}{\lambda_i(A)}\big(\widetilde{H}_{ij}\big)^2+\sum_{i\in\beta,j\in\gamma}\frac{-\lambda_j(A)}{\lambda_i(X^k)}\big(\widetilde{H}_{ij}\big)^2\\
		&+\delta_{{\cal C}_{\mathbb{S}_+^n}(X^k,Y^k)}(H),
	\end{align*}
	where ${\cal C}_{\mathbb{S}_+^n}(X^k,Y^k)=\{H\in\mathbb{S}^n\mid\widetilde{H}_{\gamma\gamma}=0\}$. It is worth to note that $d^2\delta_{\mathbb{S}_+^n}(X^k\mid Y^k)(H)$ is a generalized quadratic form by using  Definition \ref{def:genguadf}, as $d^2\delta_{\mathbb{S}_+^n}(X^k\mid Y^k)(0)=0$ and $\partial d^2\delta_{\mathbb{S}_+^n}(X^k\mid Y^k)(H)={\cal R}(H)+{\cal N}_{{\cal C}_{\mathbb{S}_+^n}(X^k,Y^k)}(H)$
	with ${\cal N}_{{\cal C}_{\mathbb{S}_+^n}(X^k,Y^k)}(H)$ being a subspace and
	$${\cal R}(H)(\Delta H)=-4\sum_{i\in\alpha,j\in\gamma}\frac{\lambda_j(A)}{\lambda_i(A)}(\widetilde{H})_{ij}\cdot(P^T\Delta HP)_{ij}-4\sum_{i\in\beta,j\in\gamma}\frac{\lambda_j(A)}{\lambda_i(A)}(\widetilde{H})_{ij}\cdot(P^T\Delta HP)_{ij}$$
	being linear on $H$.
	This also implies that $\delta_{\mathbb{S}_+^n}$ is generalized twice differentiable at $X^k$ for $Y^k$.
	
Let $f_1^k(H)=\sum_{i\in\alpha,j\in\gamma}\frac{-\lambda_j(A)}{\lambda_i(A)}\big(\widetilde{H}_{ij}\big)^2$, $f_2^k(H)=\delta_{{\cal C}_{\mathbb{S}_+^n}(X^k,Y^k)}(H)$ and $f_3^k(H)=\sum_{i\in\beta,j\in\gamma}\frac{-\lambda_j(A)}{\lambda_i(X^k)}\big(\widetilde{H}_{ij}\big)^2$. By the definition of continuous convergence mentioned in \cite[page 250]{RoWe98}, we know that $f_1^k$ converges continuously to $f_1$ with
$$f_1(H)=\sum_{i\in\alpha,j\in\gamma}\frac{-\lambda_j(A)}{\lambda_i(A)}\big(\widetilde{H}_{ij}\big)^2.$$
 It follows from \cite[Theorem 7.11]{RoWe98} that $f_1^k$  epi-converges to $f_1$. It can be checked easily that $f_1^k$ also  pointwise converges \cite[page 239]{RoWe98} to $f_1$.
Since ${\cal C}_{\mathbb{S}_+^n}(X^k,Y^k)$ is a constant closed set, we know  that $f_2^k$ pointwise converges and epi-converges \cite[Proposition 4.4]{RoWe98} to $f_2$ with $f_2(H)=\delta_{{\cal C}_{\mathbb{S}_+^n}(X^k,Y^k)}(H)$. It follows from \cite[Theorem 7.46(a)]{RoWe98} that $f_1^k+f_2^k$ epi-convergences to $f_1+f_2$. Also, it can be checked directly that $f_1^k+f_2^k$ pointwise convergences to $f_1+f_2$. By the construction of $z^k$, the sequence of $\{f_3^k\}$ is nondecreasing ($f_3^k\leq f_3^{k+1}$). We know from \cite[Proposition 7.4]{RoWe98} that $f_3^k$ epi-converges to $\sup_k\{\mbox{cl}f_3^k\}$, where $\mbox{cl}f_3^k$ is the closure of $f_3^k$. It is easy to see that $\sup_k\{\mbox{cl}f_3^k\}=\sup_k \{f_3^k\}=\delta_{\cal V}$ with ${\cal V}=\{H\in\mathbb{S}^n\mid\widetilde{H}_{\beta\gamma}=0\}$ and $f_3^k$ pointwise converges to $\delta_{\cal V}$. Since $f_1^k+f_2^k$ pointwise and epi-converges to $f_1+f_2$, by using \cite[Theorem 7.46(a)]{RoWe98} again, we obtain that $f_1^k+f_2^k+f_3^k$ epi-converges to $f_1+f_2+f_3$.

Combining the above discussion with \eqref{eq:affcsdp}, the generalized quadratic form $\frac{1}{2}d^2\delta_{\mathbb{S}_+^n}(X^k\mid Y^k)(H)$ converges epigraphically to generalized quadratic form
	$$\sum_{i\in\alpha,j\in\gamma}\frac{-\lambda_j(A)}{\lambda_i(A)}(\widetilde{H}_{ij})^2+\delta_{{\rm aff}\,{\cal C}_{\mathbb{S}_+^n}(G(\bar{x}),\overline{Y})}(H). $$
	Thus we have verified this proposition.
\end{proof}

The following result is on the explicit characterization of the strong variational sufficiency of local optimality for NLSDP, which is the main result of this paper.
\begin{theorem}\label{thm:varscequissosc}
	Let $\bar{x}\in\mathbb{X}$ be a stationary point to the NLSDP \eqref{eq:NLSDP} and $\overline{Y}\in {\cal M}(\bar{x})$. Then the following three conditions are equivalent
	\begin{itemize}
		\item[(i)] the strong variational sufficient condition with respect to $(\bar{x},\overline{Y})$ holds;
		\item[(ii)] the strong second order sufficient condition (SOSC) \eqref{eq:defssosc} holds at $(\bar{x},\overline{Y})$;
		\item[(iii)] there exist $\rho_0>0$ and $\underline{\eta}>0$ such that for any $\rho\geq \rho_0$ and any $W\in\partial_B\Pi_{\mathbb{S}_-^n}(G(\bar{x})+\rho^{-1}\overline{Y})$,
		\begin{equation*}\label{eq:psdofpb}
		\langle d, {\cal A}_{\rho}(\overline{Y},W)d\rangle\geq\underline{\eta}\|d\|^2\quad\forall d\in\mathbb{X},
		\end{equation*}
		where ${\cal A}_{\rho}(\overline{Y},W)$ is defined by \eqref{eq:algrad2}.
	\end{itemize}
\end{theorem}
\begin{proof}
	``$(i)\Longrightarrow(ii)$":  This direction can be obtained directly from Proposition \ref{prop:onesidessosc} and Proposition \ref{prop:thm5} as we only need to substitute \eqref{eq:qeqsg} into \eqref{svs-equ}.
	
	``$(ii)\Longrightarrow(iii)$": This proof sketch is similar to \cite[Proposition 4]{SSZhang08} although they require the validity of nondegeneracy, which is superfluous here.
	It follows from Definition \ref{eq:defssosc} that there exists $\eta_0>0$ such that
	$$\langle L_{xx}''(\bar{x},\overline{Y})d,d\rangle-\varUpsilon_{G(\bar{x})}\big(\overline{Y},G'(\bar{x}){d}\big)\geq\eta_0\|d\|^2\quad\forall\; G'(\bar{x})d\in{\rm aff}\,{\cal C}_{\mathbb{S}_+^n}(G(\bar{x}),\overline{Y}).$$
	By using \cite[Lemma 7]{SSZhang08}, there exist two positive numbers $\rho_1$ and $\underline{\eta}\in(0,\eta_0/2]$ such that for any $\rho'\geq \rho_1$,
	\begin{align*}
		&\langle L_{xx}''(\bar{x},\overline{Y})d,d\rangle-\varUpsilon_{G(\bar{x})}\big(\overline{Y},G'(\bar{x}){d}\big)\\
		&+\rho'\|\overline{P}^T_{\gamma}(G'(\bar{x})d)\overline{P}_{\gamma}\|^2+\rho'\|\overline{P}^T_{\beta}(G'(\bar{x})d)\overline{P}_{\gamma}\|^2\geq2\underline{\eta}\|d\|^2\quad \forall d\in\mathbb{X}.
	\end{align*}
	Suppose a sufficient large $\rho_0\geq\rho_1$. For any $\rho\geq\rho_0$ and $d\in{\cal X}$, we have
	{\small \begin{align*}
			&-\varUpsilon_{G(\bar{x})}\big(\overline{Y},G'(\bar{x}){d}\big)-2\rho\sum_{i\in\alpha,j\in\gamma}\frac{-\lambda_j}{\rho\lambda_i-\lambda_j}(\overline{P}^T(G'(\bar{x})d)\overline{P})_{ij}^2\\
			&=2\sum\limits_{i\in\alpha,\,j\in\gamma}\frac{-\lambda_j}{\lambda_i}(\overline{P}^T(G'(\bar{x}){d})\overline{P})_{ij}^2-2\rho\sum_{i\in\alpha,j\in\gamma}\frac{-\lambda_j}{\rho\lambda_i-\lambda_j}(\overline{P}^T(G'(\bar{x})d)\overline{P})_{ij}^2\\
			&=2\sum\limits_{i\in\alpha,\,j\in\gamma}\frac{\lambda_j^2}{\lambda_i(\rho\lambda_i-\lambda_j)}(\overline{P}^T(G'(\bar{x})d)\overline{P})_{ij}^2\leq\underline{\eta}\|d\|^2,
	\end{align*}}
	where the last inequality can be obtained by sufficiently large $\rho$.
	Combining the above two relations together, we have for any $\rho\geq\rho_0$,  $\rho'\geq\rho_0\geq\rho_1$,
	\begin{align*}
		&\langle L_{xx}''(\bar{x},\overline{Y})d,d\rangle+2\rho\sum_{i\in\alpha,j\in\gamma}\frac{-\lambda_j}{\rho\lambda_i-\lambda_j}(\overline{P}^T(G'(\bar{x})d)\overline{P})_{ij}^2\\
		&+\rho'\|\overline{P}^T_{\gamma}(G'(\bar{x})d)\overline{P}_{\gamma}\|^2+\rho'\|\overline{P}^T_{\beta}(G'(\bar{x})d)\overline{P}_{\gamma}\|^2\geq\underline{\eta}\|d\|^2.
	\end{align*}
	Then it can be checked directly that for all $d\in\mathbb{X}$ and $\rho\geq\rho_0$,  $\langle d,{\cal A}_{\rho}(\overline{Y},W)d\rangle\geq\underline{\eta}\|d\|^2 $.

	``$(iii)\Longrightarrow(i)$": It follows directly from \cite[Theorem 3]{Roh22} as ${\cal A}_{\rho}(\overline{Y},W)$ coincides with \cite[(3.6)]{Roh22}.
\end{proof}
\begin{remark}
{Next, we will discuss the relationship between tilt stability and variationally strong convexity/strong variational sufficient condition. 
As shown in \cite{rockvtn}, for general optimization problem, 
the variationally strong convexity implies  the tilt stability  \cite[Definition 3]{rockvtn} of local minimizer. 
However, it is worth to note that tilt stability usually does not imply variationally strong convexity as explained in \cite[Remark 2.8]{KMT22} and \cite{rockvtn}.  
For NLSDP, the general objective function of problem \eqref{eq:NLSDPc} is amenable under the Robinson constraint qualification (RCQ) \cite{Robinson76}. 
It follows from \cite[Proposition 2.9]{KMT22} that variationally strong convexity is equivalent to tilt stability under RCQ for NLSDP. For more information on the characterization of variationally strong convexity, the readers may refer to \cite{Roh22,Rockalm,KMT22}.

As an application of the tilt stability to the minimization of $\phi_{\rho}$ \eqref{eq:phir}, Rockafellar defines the so-called augmented tilt stability \cite[(2.18)]{Roh22}. In \cite[Theorem 2]{Roh22}, Rockafellar also shows augmented tilt stability is equivalent to the strong variational sufficient condition, i.e., the variationally strong convexity of $\phi_{\rho}$, without constraint qualifications. By using Theorem \ref{thm:varscequissosc}, the strong SOSC is also equivalent to the augmented tilt stability. }

\end{remark}

Recently, Khanh et al. \cite[Theorem 6.5]{KMT22} also provide a sufficient condition for (strong) variational sufficiency of strongly amenable problem in the form of \eqref{eq:genp}. However, the strong variational sufficiency mentioned in \cite{KMT22} is different from the one used here. In \cite{KMT22}, strong variational sufficiency is defined as the variationally strong convexity of the function $x\rightarrow f(x)+\theta(G(x))$, while this paper focuses on that of the perturbed function $(x,u)\rightarrow f(x)+\theta(G(x)+u)+\frac{r}{2}\|u\|^2$. 
In their result, they require a condition named second order qualification condition \cite[(3.15)]{MR12} at $\bar{x}$, i.e.,
\begin{equation}\label{eq:sec-ord-CQ}
	{\rm ker}\,G'(\bar{x})^*\cap\partial^2\theta(G(\bar{x}),\overline{Y})(0)=\{0\},
\end{equation}
where $\partial^2\theta(G(\bar{x}),\overline{Y})$ is defined in \cite[Definition 2.1]{MR12}. It follows from \cite[Theorem 3.1]{DSYe14} that \eqref{eq:sec-ord-CQ} is equivalent to the nondegeneracy \cite[Definition 3.3]{Sun06} for NLSDP. 
By combining \cite[Theorem 6.2]{KMT22}, Theorem \ref{thm:varscequissosc} and \cite[Theorem 5.6, Lemma 6.3]{fullstability}  together, we know that under the nondegenerate condition, the two strong variational sufficiencies are the same for NLSDP. However, their relationship without the nondegeneracy condition remains unclear. 

Moreover, by using \cite[Example 2.2]{USN84}, it can be proved in a similar manner to \cite[Theorem 3]{Roh22} that the variational sufficiency of local optimality for NLSDP \eqref{eq:NLSDP} is equivalent to the existence of $\rho_0>0$ and a convex open neighborhood ${\cal V}$ of $(\bar{x},\overline{Y})$ such that for any $\rho\geq \rho_0$, $(x,Y)\in{\cal V}$ and any $W\in\partial_B\Pi_{\mathbb{S}_-^n}(G(x)+\rho^{-1}{Y})$,
\begin{equation*}
\langle d, {\cal A}_{\rho}({Y},W)d\rangle\geq0\quad\forall d\in\mathbb{X}.
\end{equation*}
However, whether the variational sufficiency of local optimality is equivalent to certain second order optimality condition without any constraint qualification for NLSDP \eqref{eq:NLSDP} is still a future work that we are working on.

\section{Semi-smooth Newton-CG based ALM for nonconvex NLSDP}\label{sec:app}

In this section, we apply the main result Theorem \ref{thm:varscequissosc} to study the local convergence rate of the (extended) ALM \eqref{eq:subp1} for solving NLSDP. The detail algorithm is stated in Algorithm \ref{algo1}.  By \cite[Theorems 1.1 and 1.2]{Rockalm}, the strong variational sufficient condition with respect to $(\bar{x},\overline{Y})$ for local optimality holds if and only if there exist $\bar{\rho}>0$ and a closed convex neighborhood ${\cal X}\times{\cal Y}$ of $(\bar{x},\overline{Y})$ such that for all $\rho\geq\bar{\rho}$, ${\cal L}_{{\rho}}(x,Y)$ is strongly convex in $x\in{\cal X}$ with modulus $s>0$ $\footnote{A function $\psi:\mathbb{X}\rightarrow\mathbb{R}$ is said to be strongly convex with modulus $s>0$ if $\psi-\frac{s}{2}\|\cdot\|^2$ is convex. }$ for all $Y\in{\cal Y}$ and concave in $Y\in{\cal Y}$ for all $x\in{\cal X}$.
\begin{algorithm}
\caption{\bf Augmented Lagrangian method for solving \eqref{eq:NLSDP}}
\label{algo1}
{\small
\begin{algorithmic}[1]
\REQUIRE Let $(x^0,Y^0)\in\mathbb{X}\times\mathbb{S}^n$, $\rho^0>\bar{\rho}$. Set $k:=0$.
\STATE If $(x^k,Y^k)$ satisfies a suitable termination criterion: STOP.
\STATE Compute $x^{k+1}$ such that
\begin{equation}\label{eq:almiter}
x^{k+1}\approx\bar{x}^{k+1}=\arg\min_{x\in{\cal X}}{\cal L}_{\rho^k}(x,Y^k).
\end{equation}
\STATE Update the vector of multipliers to
\begin{equation*}\label{eq:lambda update}
Y^{k+1}:=Y^k+\widetilde{\rho}^k\big[G(x^{k+1})-\Pi_{\mathbb{S}_+^n}(G(x^{k+1})+{Y^k}/{\rho^k})\big],
\end{equation*}
where $\widetilde{\rho}^k=\rho^k-\bar{\rho}$.
\STATE Update nondecreasing positive sequence $\rho^{k+1}$ according to certain rules.
\STATE Set $k\leftarrow k+1$ and go to 1.
\end{algorithmic}}
\end{algorithm}

In the above algorithm, subproblem \eqref{eq:almiter} is solved inexactly. Three increasing tightness stopping criteria for the updating of $x^{k+1}$ are illustrated in \cite[equation (1.15)]{Rockalm}:
\begin{equation}\label{eq:stopcrit}
\big(2\widetilde{\rho}^k[{\cal L}_{\rho^k}(x^{k+1},Y^k)-\inf_{\cal X}{\cal L}_{\rho^k}(\cdot,Y^k)]\big)^{1/2}\leq
\begin{cases}
(a) \;\epsilon_k,\\
(b) \;\epsilon_k\min\{1,\|\widetilde{\rho}^k({\cal L}_{\rho^k})_Y'(x^{k+1},Y^k)\|\},\\
(c) \;\epsilon_k\min\{1,\|\widetilde{\rho}^k({\cal L}_{\rho^k})_Y'(x^{k+1},Y^k)\|^2\}
\end{cases}
\end{equation}
$$\mbox{with}\; \epsilon_k\in(0,1) \; \sum_{k=0}^{\infty}\epsilon_k=\sigma<\infty\;{\rm and}\; \rho^k\rightarrow\rho^{\infty}\leq\infty.$$
It is worth to note that $(c)$ is first introduced in \cite{Roh21} to support linear convergence in partnership with strong variational sufficiency.
It follows from \cite[Theorem 3.1]{Rockalm} that \eqref{eq:stopcrit} can be replaced by
\begin{equation}\label{eq:stopcritrep}
\sqrt{\widetilde{\rho}^k}\|({\cal L}_{\rho^k})_x'(x^{k+1},Y^k)\|\leq
\begin{cases}
(a) \;\epsilon_k', \\
(b) \;\epsilon_k'\min\{1,\|\widetilde{\rho}^k({\cal L}_{\rho^k})_Y'(x^{k+1},Y^k)\|\}, \\
(c) \;\epsilon_k'\min\{1,\|\widetilde{\rho}^k({\cal L}_{\rho^k})_Y'(x^{k+1},Y^k)\|^2\},
\end{cases}
\end{equation}
where $\epsilon_k'=\epsilon_k\sqrt{s}$, as the strong convexity of ${\cal L}_{\rho^k}(\cdot,Y^k)$ with modulus $s$ guarantees ${\cal L}_{\rho^k}(x^{k+1},Y^k)-\inf_{\cal X}{\cal L}_{\rho^k}(\cdot,Y^k)\leq\frac{1}{2s}\|({\cal L}_{\rho^k})_x'(x^{k+1},Y^k)\|^2$.

By applying the local duality, which comes from the strong variational sufficient condition through \cite[Theorem 1]{Roh22}, we suppose $S_{KKT}(0,0)\cap{\cal X}\times{\cal Y}\neq\varnothing$,
where ${\cal X}\times{\cal Y}$ is the neighborhood mentioned at the beginning of this section. As mentioned in \cite[page 9-10]{Rockalm}, we can define the associated local primal and dual problems to ${\cal X}\times{\cal Y}$ in the following sense. The associated local primal problem is
\begin{equation}\label{eq:widehatp}
\min\; \widehat{f}(x):=\sup_{Y\in{\cal Y}}{\cal L}_{\bar{\rho}}(x,Y)\quad\mbox{over}\quad x\in{\cal X}.
\end{equation}
The associated local dual problem is
\begin{equation}\label{eq:widehatd}
\max\; \widehat{h}(Y):=\inf_{x\in{\cal X}}{\cal L}_{\bar{\rho}}(x,Y)\quad\mbox{over}\quad Y\in{\cal Y}.
\end{equation}
In \cite[Theorem 2.1]{Rockalm}, the author reveals the connection between problems \eqref{eq:NLSDP}, \eqref{eq:widehatp} and \eqref{eq:widehatd}, which is of great use in the following discussion.

\subsection{Convergence analysis of ALM}
By taking advantage of Theorem \ref{thm:varscequissosc} and the genius work \cite{Rockalm}, we can view the convergence analysis of ALM for NLSDP as a direct extension.
To explore this topic, we need the definition of bounded linear regularity of a collection of closed convex sets, which can be found in, e.g., \cite[Definition 5.6]{BBorwein96}.
\begin{definition}
	Let $D_{1}, D_{2}, \ldots, D_{m} \subseteq \mathbb{X}$ be closed convex sets for some positive integer $m .$ Suppose that $D:=D_{1} \cap D_{2} \cap \ldots \cap D_{m}$ is non-empty. The collection $\left\{D_{1}, D_{2}, \ldots, D_{m}\right\}$ is said to be boundedly linearly regular if for every bounded set $\mathcal{B} \subseteq \mathbb{X}$, there exists a constant $\kappa>0$ such that
	$$
	\operatorname{dist}(x, D) \leqslant \kappa \max \left\{\operatorname{dist}\left(x, D_{1}\right), \ldots, \operatorname{dist}\left(x, D_{m}\right)\right\}\quad  \forall x \in \mathcal{B}.
	$$
\end{definition}
A sufficient condition to guarantee the property of bounded linear regularity is established in \cite[Corollary 3]{BBLi99}.
Denote
\begin{equation}\label{eq:g1g2}
	{\cal G}_1(\bar{x})=\{Y\in\mathbb{S}^n\mid f'(\bar{x})+G'(\bar{x})^*Y=0\}\;\;{\rm and}\;\; {\cal G}_2(\bar{x})=\{Y\in\mathbb{S}^n\mid Y\in{\cal N}_{\mathbb{S}_+^n}(G(\bar{x}))\}.
\end{equation}
It is easy to see that ${\cal G}_1(\bar{x})$ is a polyhedron and ${\cal G}_2(\bar{x})$ is convex.

For a stationary point $\bar{x}$ of NLSDP \eqref{eq:NLSDP}, define
\begin{equation}\label{eq:quant}
	\xi( G'(\bar{x}))=\min\{\| G'(\bar{x})^*\eta\| : \eta\in {\cal G}_1(\bar{x})^{\perp},\ \|\eta\|=1\}.
\end{equation}
The following condition is adopted from \cite{Rockalm} as it is essential in verifying the convergence rate.
\begin{asp}\label{asp1}
	There exist $b>0$ and $\varepsilon>0$ such that the local dual problem \eqref{eq:widehatd} satisfies
	$\widehat{h}(Y)\leq\max_{\cal Y}\widehat{h}-b\,{\rm dist}^2(Y,{\cal H})\; \mbox{when} \quad\| Y-\overline{Y}\|<\varepsilon$,
	where ${\cal Y}$ is the closed convex neighborhood of $\overline{Y}$ given at the beginning of Section \ref{sec:app}, ${\cal H}:=\arg\max_{\cal Y}\widehat{h}$.
\end{asp}

The following result, which is originally proposed in \cite[Theorem 4.2]{Rockalm} for polyhedral case, provides the sufficiency of  Condition \ref{asp1}. By applying the boundedly linear regularity \cite{BBLi99}, one can obtain the results in similar approach as that of \cite[Theorem 4.2]{Rockalm}, directly. We omit it here for simplicity.  Moreover, it follows from \cite[page 34]{Rockalm} that Condition \ref{asp1} trivially holds if $G'(\bar{x})=0$.
\begin{proposition}\label{thm:c1h}
	Let $\bar{x}\in\mathbb{X}$ be a stationary point to the NLSDP \eqref{eq:NLSDP} and $\overline{Y}\in {\cal M}(\bar{x})$, where ${\cal M}(\bar{x})$ is given in \eqref{eq:defcalm}. Suppose $G'(\bar{x})\neq0$. If strong SOSC with respect to $(\bar{x},\overline{Y})$ holds and the collection $\{{\cal G}_1(\bar{x}),{\cal G}_2(\bar{x})\}$ is boundedly linearly regular, where ${\cal G}_1(\bar{x})$ and ${\cal G}_2(\bar{x})$ are given in \eqref{eq:g1g2},
	then we have $\xi(G'(\bar{x}))>0$, where $\xi(G'(\bar{x}))$ is defined by \eqref{eq:quant} and Condition \ref{asp1} holds for
	\begin{equation}\label{eq:partp12}
		b=\frac{\kappa}{a_2+a_1}\;\;\mbox{with}\;\;a_2=b_0^{-1}+2\bar{\rho}\;\;\mbox{and}\;\;a_1=\frac{2\|L_{xx}''(\bar{x},\overline{Y})+\hat{\rho}I\|}{\xi(G'(\bar{x}))^2},
	\end{equation}
	where $\hat{\rho}=\lambda_{\max}(\bar{\rho} G'(\bar{x})^* G'(\bar{x}))+\varepsilon'$, $\varepsilon'$ is some positive constant and $b_0$ is the quadratic parameter $\kappa$ given in  \cite[Proposition 2.1]{CSToh2016a}.
\end{proposition}

\begin{remark}
	It follows from \cite[Proposition 3.2]{CSToh2016a} (see also \cite[Proposition 17]{CDZhao}) that the bounded linear regularity of the collection $\{{\cal G}_1(\bar{x}),{\cal G}_2(\bar{x})\}$ holds under one of the following two conditions:
	\begin{itemize}
		\item[(i)] ${\cal G}_2(\bar{x})$ is a polyhedron, i.e., $\lvert\gamma\rvert\geq n-1$, where $\gamma$ is the negative eigenvalue index set of matrix $G(\bar{x})+\overline{Y}$;
		\item[(ii)] there exists a strict complementarity KKT pair $(\bar{x}, \widetilde{Y})$ ($\widetilde{Y}$ does not have to be $\overline{Y}$), i.e., ${\rm rank}(G(\bar{x}))+{\rm rank}(\widetilde{Y})=n$.
	\end{itemize}
\end{remark}

Next, we shall present the local convergence result of ALM for NLSDP. We say a sequence $z^k>0$ converges Q-linearly to 0 at a rate $c$ if $\displaystyle\lim\sup_{k\rightarrow\infty}\frac{z^{k+1}}{z^k}\leq c<\infty$. When $c=0$, we say $z^k$ converges Q-superlinearly to 0. Moreover, a sequence $y^k>0$ converges R-linearly to 0 at a rate $c$ if $y^k\leq z^k$ with $z^k>0$  converges Q-linearly to 0 at that rate. The following closedness condition relative to the closed convex set ${\cal M}(\bar{x})$ is taken from \cite[Theorem 2.2]{Rockalm}.  Recall that ${\cal X}$, ${\cal Y}$ are the closed convex neighborhood of $\bar{x}$, $\overline{Y}$ mentioned in the beginning of Section \ref{sec:app}.
\begin{asp}\label{cond:closedc}
	We say the initial point $Y^0$ and $\sigma>0$ in \eqref{eq:stopcrit} satisfies the following closedness condition relative to the closed convex set ${\cal M}(\bar{x})$ \eqref{eq:defcalm} if there exists $\eta>{\rm dist}(Y^0,{\cal M}(\bar{x}))+\sigma$ such that
	$$\left\{Y\mid\| Y-Y^0\|\leq3\eta\right\}\subset {\cal Y}. $$
\end{asp}
It is easy to see that this condition indicates $Y^0$ to be sufficiently close to ${\cal M}(\bar{x})$ and $\overline{Y}$. Also, the computations of subproblems at each iteration need to be sufficiently accurate. If subproblems are solved exactly, a sufficient condition for the above one is that there exists $\eta>0$ such that ${\cal B}_{\eta}(Y^0)\subseteq{\rm int}\,{\cal Y}$ and ${\rm dist}(Y^0,{\cal M}(\bar{x}))\leq\eta/3$.

By using Theorem \ref{thm:varscequissosc} and \cite[Theorem 2.2, 2.3, 3.1, 3.2]{Rockalm}, we immediately get the following local convergence result of ALM for solving NLSDP.
\begin{theorem}\label{thm:almrate}
	Let $\bar{x}\in{\cal X}$ be a stationary point to the NLSDP \eqref{eq:NLSDP} and $\overline{Y}\in {\cal M}(\bar{x})$. Suppose the strong SOSC \eqref{eq:defssosc} holds at $(\bar{x},\overline{Y})$.  Let the initial point $Y^0$ and $\sigma$ in \eqref{eq:stopcrit} satisfy Condition \ref{cond:closedc}. Suppose the set $\{x\mid-({\cal L}_{\bar{\rho}})_x'(x,Y)\in{\cal N}_{\cal X}(x)\}$ is nonempty and bounded when $Y\in{\rm int}\,{\cal Y}$.
	\begin{itemize}
		\item[(i)] Under stopping criterion (\ref{eq:stopcritrep}\,a),  we have the sequence $\{Y^k\}$ converges within ${\rm int}\,{\cal Y}$ to a particular $\widehat{Y}\in{\rm int}\,{\cal Y}$. Moreover, both $x^k$ and $\bar{x}^k$ converge to $\bar{x}$.
		\item[(ii)] Stopping criterion in (i) is strengthened into (\ref{eq:stopcritrep}\,b) and suppose $\{{\cal G}_1(\bar{x}),{\cal G}_2(\bar{x})\}$ is also boundedly linearly regular, where ${\cal G}_1(\bar{x})$ and ${\cal G}_2(\bar{x})$ are given in \eqref{eq:g1g2}.
		Then we have ${\rm dist}(Y^k,{\cal M}(\bar{x}))\rightarrow0$ with
		$${\rm dist}(Y^{k+1},{\cal M}(\bar{x}))\leq\frac{1}{\sqrt{1+{b}^2(\rho^{\infty})^2}}{\rm dist}(Y^k,{\cal M}(\bar{x}))$$
		and $\bar{x}^k\rightarrow\bar{x}$ with
		$$\|\bar{x}^k-\bar{x}\|\leq\frac{1}{s}{\rm dist}(Y^k,{\cal M}(\bar{x})),$$
		where $\bar{x}^k$ is the exact solution of subproblems in \eqref{eq:almiter}, $b$ is given in Condition \ref{asp1} and $s$ is the strong convexity modulus of ${\cal L}_{\rho}(x,Y)$ on $x$ with $\rho\geq\bar{\rho}$.
		\item[(iii)] If stopping criterion in (ii) is strengthened into (\ref{eq:stopcritrep}\,c), we have
		$Y^k\rightarrow\widehat{Y}$ with
		$$\|Y^{k+1}-\widehat{Y}\|\leq\frac{1}{\sqrt{1+{b}^2(\rho^{\infty})^2}}\|Y^{k}-\widehat{Y}\|.$$
		Moreover, if the stopping criterion is further supplemented by
		$$\|({\cal L}_{\rho^k})_x'(x^{k+1},Y^k)\|\leq c\|Y^{k+1}-Y^k\|\;\mbox{for some fixed}\;c, $$
		we have $x^k\rightarrow\bar{x}$ with
		$$\|{x}^k-\bar{x}\|\leq p\|Y^{k}-\widehat{Y}\|$$
		for some $p>0$.
	\end{itemize}
\end{theorem}

As illustrated in \cite[Theorem 2.3]{Rockalm}, the condition ``set $\{x\mid-({\cal L}_{\bar{\rho}})_x'(x,Y)\in{\cal N}_{\cal X}(x)\}$ is nonempty and bounded when $Y\in{\rm int}\,{\cal Y}$" can be reduced to ``the existence of $Y\in{\rm int}\,{\cal Y}$ such that $\{x\mid-({\cal L}_{\bar{\rho}})_x'(x,Y)\in{\cal N}_{\cal X}(x)\}$ being nonempty and bounded". This condition is trivially satisfied as ${\cal X}$ is a neighborhood of $\bar{x}$ and $(\bar{x},\overline{Y})$ belongs to the set.
It is worth to note that as illustrated in  \cite[Theorem 2.3]{Rockalm}, the result in Theorem \ref{thm:almrate} (i) only requires variational sufficiency. Under variational sufficiency, the convergence of $x^k$ to $\bar{x}$ can not be obtained.  Meanwhile, under the stopping criterion (\ref{eq:stopcritrep}\,b), we may not be able to obtain from Theorem \ref{thm:almrate} (ii) the convergence of the primal iteration sequence $\{x^k\}$, since the exact solution $\bar{x}^k$ of subproblems in \eqref{eq:almiter} $\bar{x}^k$ is unknown in practice. Next, we shall show that the KKT residual of NLSDP \eqref{eq:NLSDP} also converges R-linearly, which means that the KKT residual can be used as a verifiable stopping criterion for ALM. Its proof sketch is inspired by \cite[Theorem 2]{CSToh2016}.
	\begin{proposition}\label{cor:resrate}
		Suppose the conditions in Theorem \ref{thm:almrate} (ii) hold. Define the following residual function
		\begin{equation}\label{eq:residdef}
			R(x,Y):=\|L_x'(x,Y)\|+\|G(x)-\Pi_{\mathbb{S}_+^n}(G(x)+Y)\|.
		\end{equation}
		Then, for $k$ sufficiently large, if $\sqrt{\widetilde{\rho}^k}\epsilon_k<1$, we have there exists $a>0$ such that
		$$R(x^{k+1},Y^{k+1})\leq c^k{\rm dist}(Y^k,{\cal M}(\bar{x}))$$
		with $c^k=(\epsilon_k'\sqrt{\widetilde{\rho}^k}+(1+\bar{\rho}+a\bar{\rho})(\widetilde{\rho}^k)^{-1})(1-\sqrt{\widetilde{\rho}^k}\epsilon_k)^{-1}$.
	\end{proposition}
	\begin{proof} Let $\widehat{Y}^{k+1}=Y^k+\rho^k(G(x^{k+1})-\Pi_{\mathbb{S}_+^n}(G(x^{k+1})+(\rho^k)^{-1}Y^k))$ and $z^{k+1}=\Pi_{\mathbb{S}_+^n}(G(x^{k+1})+(\rho^k)^{-1}Y^k)$. We have $G(x^{k+1})-z^{k+1}=(\widetilde{\rho}^k)^{-1}(Y^{k+1}-Y^k)$.  It follows directly from (\ref{eq:stopcritrep}\,b) that  for each $k$,
		\begin{equation*}
			\|L_x'(x^{k+1},\widehat{Y}^{k+1})\|=\|({\cal L}_{\rho^{k}})_x'(x^{k+1},Y^{k})\|\leq \epsilon_k'\sqrt{\widetilde{\rho}^k}\|Y^{k+1}-Y^k\|.
		\end{equation*}
		Then, there exists $a>0$ such that
	{\footnotesize	\begin{align}\label{eq:resipf1}
			&\|L_x'(x^{k+1},{Y}^{k+1})\|=\|L_x'(x^{k+1},\widehat{Y}^{k+1})+G'(x^{k+1})^*(Y^{k+1}-\widehat{Y}^{k+1})\|\nonumber\\
			&\leq\|L_x'(x^{k+1},\widehat{Y}^{k+1})\|+\|G'(x^{k+1})^*(Y^{k+1}-\widehat{Y}^{k+1})\|%\nonumber\\
			\leq\|L_x'(x^{k+1},\widehat{Y}^{k+1})\|+a\|Y^{k+1}-\widehat{Y}^{k+1}\|\nonumber\\
			&=\|L_x'(x^{k+1},\widehat{Y}^{k+1})\|+a\bar{\rho}\|G(x^{k+1})-\Pi_{\mathbb{S}_+^n}(G(x^{k+1})+(\rho^k)^{-1}Y^k)\|\nonumber\\
			&=\|L_x'(x^{k+1},\widehat{Y}^{k+1})\|+a\bar{\rho}(\widetilde{\rho}^k)^{-1}\|Y^{k+1}-Y^k\|%\nonumber\\
			\leq(\epsilon_k'\sqrt{\widetilde{\rho}^k}+a\bar{\rho}(\widetilde{\rho}^k)^{-1})\|Y^{k+1}-Y^k\|,
		\end{align}}
		where the second inequality follows from the twice differentiable continuity and the boundedness of $x^k$ obtained from Theorem \ref{thm:almrate}(a).
		It can be verified directly from \cite[Theorem 2.26]{RoWe98} that $\widehat{Y}^{k+1}\in\partial\delta_{\mathbb{S}_+^n}(z^{k+1})$ and
		\begin{align*}
			&\|G(x^{k+1})-\Pi_{\mathbb{S}_+^n}(G(x^{k+1})+\widehat{Y}^{k+1})\|\nonumber\\
			&=\|G(x^{k+1})-\Pi_{\mathbb{S}_+^n}(G(x^{k+1})+\widehat{Y}^{k+1})\|-\|z^{k+1}-\Pi_{\mathbb{S}_+^n}(\widehat{Y}^{k+1}+z^{k+1})\|\nonumber\\
			&\leq\|G(x^{k+1})-\Pi_{\mathbb{S}_+^n}(G(x^{k+1})+\widehat{Y}^{k+1})-(z^{k+1}-\Pi_{\mathbb{S}_+^n}(\widehat{Y}^{k+1}+z^{k+1}))\|\nonumber\\
			&\leq\|G(x^{k+1})-z^{k+1}\|=(\widetilde{\rho}^k)^{-1}\|Y^{k+1}-Y^k\|.
		\end{align*}
		It then follows that
		\begin{align}
			&\|G(x^{k+1})-\Pi_{\mathbb{S}_+^n}(G(x^{k+1})+{Y}^{k+1})\|\nonumber\\
			&\leq\|G(x^{k+1})-\Pi_{\mathbb{S}_+^n}(G(x^{k+1})+\widehat{Y}^{k+1})\|+\|\Pi_{\mathbb{S}_+^n}(G(x^{k+1})+\widehat{Y}^{k+1})-\Pi_{\mathbb{S}_+^n}(G(x^{k+1})+{Y}^{k+1})\|\nonumber\\
			&\leq(\widetilde{\rho}^k)^{-1}\|Y^{k+1}-Y^k\|+\|\widehat{Y}^{k+1}-Y^{k+1}\|\nonumber\\
			&=(\widetilde{\rho}^k)^{-1}\|Y^{k+1}-Y^k\|+\bar{\rho}\|G(x^{k+1})-\Pi_{\mathbb{S}_+^n}(G(x^{k+1})+(\rho^k)^{-1}Y^k)\|\nonumber\\
			&=(\widetilde{\rho}^k)^{-1}\|Y^{k+1}-Y^k\|+\bar{\rho}(\widetilde{\rho}^k)^{-1}\|Y^{k+1}-Y^k\|=(1+\bar{\rho})(\widetilde{\rho}^k)^{-1}\|Y^{k+1}-Y^k\|.\label{eq:resipf2}
		\end{align}
		Combining \eqref{eq:resipf1} and \eqref{eq:resipf2} together, we obtain
		\begin{equation}\label{eq:resipf3}
			R(x^{k+1},Y^{k+1})\leq(\epsilon_k'\sqrt{\widetilde{\rho}^k}+(1+\bar{\rho}+a\bar{\rho})(\widetilde{\rho}^k)^{-1})\|Y^{k+1}-Y^k\|.
		\end{equation}
		Then we will prove $\|Y^{k+1}-Y^k\|\leq{\rm dist}(Y^k,{\cal M}(\bar{x}))$. Let $P_k(Y^k)=\arg\max\{\widehat{h}(Y)-\frac{1}{2\widetilde{\rho}^k}\|Y-Y^k\|^2\}$. Then, we have
		\begin{align}
			&\|Y^{k+1}-Y^k\|\leq\|Y^{k+1}-P_k(Y^k)\|+\|P_k(Y^k)-Y^k\|\nonumber\\
			&\leq(2\widetilde{\rho}^k({\cal L}_{\rho^k}(x^{k+1},Y^k)-\inf_{\cal X}{\cal L}_{\rho^k}(\cdot,Y^k)))^{1/2}+\|P_k(Y^k)-Y^k\|\nonumber\\
			&\leq\sqrt{\widetilde{\rho}^k}\epsilon_k\|Y^{k+1}-Y^k\|+\|P_k(Y^k)-Y^k\|,\label{eq:resipf4}
		\end{align}
		where the second inequality follows from \cite[(2.19)]{Rockalm} and the last one follows from (\ref{eq:stopcritrep}\,b) and the paragraph under \eqref{eq:stopcritrep}. By using \cite[Proposition 1(b), the proof of Lemma 3]{CSToh2016} and \cite[Theorem 2.1]{Rockalm}, we have
		\begin{equation}\label{eq:resipf5}
			\|P_k(Y^k)-Y^k\|\leq{\rm dist}(Y^k,{\cal M}_D),
		\end{equation}
		where ${\cal M}_D$ denotes the solution set of \eqref{eq:widehatd}. It follows from \cite[Theorem 2.1]{Rockalm} that ${\rm dist}(Y^k,{\cal M}_D)={\rm dist}(Y^k,{\cal M}(\bar{x})\cap{\cal Y})={\rm dist}(Y^k,{\cal M}(\bar{x}))$ since for $k$ sufficiently large, $Y^k\in{\rm int}\,{\cal Y}$. Based on \eqref{eq:resipf3}-\eqref{eq:resipf5}, we obtain that for $k$ sufficiently large,
		$$R(x^{k+1},Y^{k+1})\leq(\epsilon_k'\sqrt{\widetilde{\rho}^k}+(1+\bar{\rho}+a\bar{\rho})(\widetilde{\rho}^k)^{-1})(1-\sqrt{\widetilde{\rho}^k}\epsilon_k)^{-1}{\rm dist}(Y^k,{\cal M}(\bar{x})).$$
		This completes the proof.
	\end{proof}
	
	\begin{remark}
		We compare our results with existing ALM convergence results for nonconvex non-polyhedral problems. \cite{KSteck19} justified the primal-dual linear convergence of ALM under SOSC and strong Robinson constraint qualification (SRCQ) for ${\cal C}^2$-cone reducible constrained problems, which include NLSDP and NLSOC. \cite{SSZhang08} proved the convergence rate of NLSDP under strong SOSC  together with nondegeneracy. \cite{WDing21} does provide the linear rate under SOSC and semi-isolated calmness of the KKT pair without requiring the multiplier to be unique. But some other assumptions are also needed. In this paper, to obtain the Q-linear convergence for multiplier and R-linear convergence for primal variable, we assume strong SOSC (Definition \ref{def:sigt}), which is much stronger than the aforementioned SOSC. However, we do not assume any restriction on the dual variable.
		
		In \cite[Example 5.3]{Rockalm}, the author also studied the ALM convergence for  second order cone programming when $G(\bar{x})\neq0$. Moreover, \cite{HMSarabi} gives the primal-dual linear convergence for NLSOC when the multiplier is unique while they only require SOSC instead of strong SOSC. By using Proposition \ref{thm:varscequissoscsoc}, without any constraint qualification, the local convergence results of ALM for NLSOC can be obtained immediately.
	\end{remark}

\subsection{Solving ALM subproblem: semi-smooth Newton-CG method}
Although the convergence properties of ALM have been established, it is equally important to study how to solve the subproblem \eqref{eq:almiter}. To guarantee the convergence rate of ALM, we need to employ the stopping criterion (\ref{eq:stopcritrep}\,$b$ or $c$). To meet the requirement of this stopping criterion, we consider using the semi-smooth Newton-CG method to solve the subproblem as it possesses the quadratic convergence rate under suitable conditions.

It follows from the proposed characterization  of strong variational sufficiency (Theorem \ref{thm:varscequissosc}) that for NLSDP \eqref{eq:NLSDP} if the strong SOSC \eqref{eq:defssosc} holds then there exists a neighborhood $\mathcal{B}_r(\bar{x},\overline{Y})$ of $(\bar{x},\overline{Y})$ such that for all $(x,Y)\in\mathcal{B}_r(\bar{x},\overline{Y})$, every elements in $\pi_x\partial_B\big(({\cal L}_{\rho})_x'\big)(x,Y)$ is positive definite. Thus,
the validity of semi-smooth Newton-CG method to solve the subproblems \eqref{eq:almiter} is ensured. The algorithm to solve the $(k+1)$-th subproblem is stated below (see Algorithm \ref{algo2}).
\begin{algorithm}
	\caption{\bf Semi-smooth Newton-CG method for solving \eqref{eq:almiter}}
	\label{algo2}
	{\small
	\begin{algorithmic}[1]
		\REQUIRE  Set the initial point $x_0=x^k$, where $x^k$ is obtained from the $k$-th iteration of Algorithm \ref{algo1}. Let $\mu\in(0,1/2)$, $\tau\in(0,1]$, $\tau_1, \tau_2, \bar{\nu}\in(0,1)$ and $\theta\in(0,1)$. Set $j:=0$.
		\STATE Given a maximum number of CG iterations $t_j>0$ and compute
		$$\nu_j=\min\{\bar{\nu},\|({\cal L}_{\rho^k})_x'(x_j,Y^k)\|^{1+\tau}\}. $$
		\STATE Choose $W_j\in\partial\Pi_{\mathbb{S}_-^n}(G(x_j)+(\rho^k)^{-1}Y^k)$. Let $V_j=L_{xx}''(x_j,\Pi_{\mathbb{S}_-^n}(G(x_j)+(\rho^k)^{-1}Y^k)+\rho^kG'(x_j)^*W_jG'(x_j)$ and $\varepsilon_j=\tau_1\min\{\tau_2,\|\nabla_x{\cal L}_{\rho^k}(x_j,Y^k)\|\}$. Apply the CG algorithm (CG($\nu_j,t_j$)) mentioned in \cite[Algorithm 1]{ZSToh} to find an approximate solution $d_j\in\mathbb{X}$ to
		$$(V_j+\varepsilon_jI)d_j=-({\cal L}_{\rho^k})_x'(x_j,Y^k) $$
		such that
		$$\|(V_j+\varepsilon_jI)d_j+({\cal L}_{\rho^k})_x'(x_j,Y^k)\|\leq \nu_j.  $$
		\STATE Set $\zeta_j=\theta^{m_j}$, where $m_j$ is the first non-negative number such that
		$${\cal L}_{\rho^k}(x_j+\theta^{m_j}d_j,Y^k)\leq{\cal L}_{\rho^k}(x_j,Y^k)+\mu\theta^{m_j}\langle ({\cal L}_{\rho^k})_x'(x_j,Y^k),d^j\rangle. $$
		\STATE Set $x_{j+1}=x_j+\zeta_jd_j$ and $j=j+1$.
	\end{algorithmic}}
\end{algorithm}
The convergence analysis framework for Algorithm \ref{algo2} is well-known \cite[Theorem 3.2]{QSun93} (see also \cite[Theorem 3.4, 3.5]{ZSToh}). We omit the detailed proof here for simplicity.
\begin{proposition}\label{cor:cgnconv}
	Suppose the strong SOSC \eqref{eq:defssosc} holds at $(\bar{x},\overline{Y})\in S_{KKT}(0,0)$. Then, Algorithm \ref{algo2} is well-defined and any accumulation point $\hat{x}$ of $\{x_j\}$ generated by Algorithm \ref{algo2} is the optimal solution to the subproblem \eqref{eq:almiter}. Furthermore, suppose that at each step $j$ when CG terminates, i.e.,
	$$\|(V_j+\varepsilon_jI)d_j+({\cal L}_{\rho^k})_x'(x_j,Y^k)\|\leq \nu_j.$$
	Then the whole sequence $\{x_j\}$ converges to $\hat{x}$ and
	$$\|x_{j+1}-\hat{x}\|=O(\|x_{j}-\hat{x}\|^{1+\tau}). $$
\end{proposition}
\begin{remark}
	It is worth noting that for convex NLSDP, the convergence result of ALM can be established under weaker conditions such that the optimal $\bar{x}$ can be non-unique (see \cite[Theorem 20]{CDZhao}). This implies that when the problem is reduced to convex case, Theorem \ref{thm:almrate} is much weaker than \cite[Theorem 20]{CDZhao} as it requires strong SOSC, which implies the local uniqueness of  $\bar{x}$.
	However, the semi-smooth Newton algorithm may fail to solve the subproblem in the absence of strong SOSC since it is equivalent to the positive definiteness of the generalized Hessian of Newton equation of the subproblem \eqref{eq:almiter} (Theorem \ref{thm:varscequissosc}). Thus strong SOSC seems to be not only sufficient to the local fast linear convergence rate of ALM, but also necessary for the invertibility of generalized Hessian of augmented Lagrangian function for NLSDP, which is crucial for the semi-smooth Newton CG method for solving the ALM subproblem \eqref{eq:almiter}. 
\end{remark}

\section{Numerical experiments}\label{sec:num}
Consider the following optimization problem
\begin{equation}\label{3times3}
	\begin{array}{cl}
		\displaystyle{\min_{X\in \mathbb{S}^{n}}} & \frac{1}{2}\langle X,Q\circ X\rangle\\ [3pt]
		{\rm s.t.} &X\in \mathbb{S}^n_+\\
		& B\circ X=0,
	\end{array}
\end{equation}
where ``$\circ$" denotes the Hadamard product of matrices, i.e., for any $U$ and $V\in\mathbb{R}^{p\times q}$, $(U\circ V)_{ij}=U_{ij}V_{ij}$,
{\footnotesize
	$$
	B=\left[\begin{array}{cccc}
		0 & \cdots & 0 & 1\\
		\vdots & \ddots & \vdots & \vdots\\
		0 & \cdots & 0 & 1 \\
		1  & \cdots & 1 & 1
	\end{array}\right]\quad \mbox{and} \quad
	Q=\left[\begin{array}{ccccc}
		q & 1 &  \cdots & 1 & 0\\
		1 & q &  \cdots & 1  & 0 \\
		\vdots &  \vdots & \ddots  & \vdots & \vdots\\
		1 & 1 &  \cdots & q & 0\\
		0 &  0 & \cdots & 0 & -1
	\end{array}\right]
	$$  }
with a given $q\geq n-1$.
The Lagrangian function of \eqref{3times3} is given by
$$L(X,Y,Z)=\frac{1}{2}\langle X,Q\circ X\rangle+\langle X, Y\rangle+\langle B\circ X,Z\rangle.$$
It can be checked directly that $\overline{X}=0$ is a local optimal solution with the multiplier $\overline{Y}={\rm Diag}(0,\dots,0,-1)$ and $\overline{Z}={\rm Diag}(0,\dots,0,1)$. In fact, the corresponding multiplier set of $\overline{X}$ is given by
$${\cal M}(\overline{X})=\{(Y,Z)\in\mathbb{S}^n\times\mathbb{S}^n\mid Y+B\circ Z=0, Y\in\mathbb{S}^n_-\}.$$
As the problem \eqref{toyexample} mentioned in Introduction, it follows from  \cite[Theorem 4.1]{ZKurcyusz79} that the Robinson constraint qualification \cite{Robinson76} does not hold at $\overline{X}$, due to the unboundedness of ${\cal M}(\bar{x})$.
It is clear that $L_{XX}''(\overline{X},\overline{Y},\overline{Z})=Q$, which is positive definite over $d\in\{d\in\mathbb{S}^n\mid B\circ d=0, d\in{\rm aff}\,{\cal C}_{\mathbb{S}_+^n}(\overline{X},\overline{Y})\}$, implies the validity of strong SOSC \eqref{eq:defssosc}. Also, boundedly linear regularity is satisfied at $(\overline{X},\overline{Y},\overline{Z})$ since $(Y,Z)\in{\cal G}_1(\overline{X})\cap{\cal G}_2(\overline{X})$ with $Y={\rm Diag}(-1,\dots,-1)$, $Z={\rm Diag}(0,\dots,0,1)$.

Next, we shall apply Algorithm \ref{algo1} to solve problem \eqref{3times3} with different dimensions. The subproblem \eqref{eq:almiter} is solved by Algorithm \ref{algo2} where the exactness in \eqref{eq:stopcritrep} is chosen as $\epsilon_k'=0.01\times(1/1.05)^{k-1}$ and the stopping criterion (\ref{eq:stopcritrep}\,b) is employed. The algorithm is stopped when the KKT residual $R(x^{k},Y^k)$ defined in \eqref{eq:residdef} is less than {\tt 1e-5}. The codes are implemented in Matlab (R2018a), and the numerical experiments are run under a 64-bit MacOS on an Intel Cores i5 2.3GHz CPU with 8GB memory.
The following table (Table \ref{tab}) shows the numerical results of different dimensions of \eqref{3times3}.
\begin{table}[h]\label{tab}
	\center
	{\small\begin{tabular}{ccccc}
		\toprule %[2pt]
		$n$ & $q$ & iteration &  KKT residual & cpu(s)   \\
		\midrule %[2pt]
		3 & 2 & 8 & 8.27e-06 & 0.22  \\
		100 & 200 & 11 & 8.98e-06 & 3.11  \\
		1000 & 1500 & 21 & 9.57e-06 &1083.34  \\
		\bottomrule%[2pt]
	\end{tabular}}
	\caption{Numerical results of semi-smooth Newton-CG based ALM for problem \eqref{3times3}.}
\end{table}
Noting that the distance from $(Y^k,Z^k)$ to ${\cal M}(\overline{X})$ is difficult to compute, we use the following alternative ${\rm dist}((Y^k,Z^k),{\cal M}(\overline{X}))=O( \|Y^k+B\circ Z^k\|+\|\Pi_{\mathbb{S}_+^n}(Y^k)\|)$ as boundedly linear regularity is satisfied. The detail iterative performance of ALM for solving problem \eqref{3times3} with $n=1000$ and $q=1500$ is also reported in Figure \ref{fig4}.
\begin{figure}[h]
	\center
	\scriptsize
	\begin{tabular}{cc}
		\includegraphics[scale=0.16]{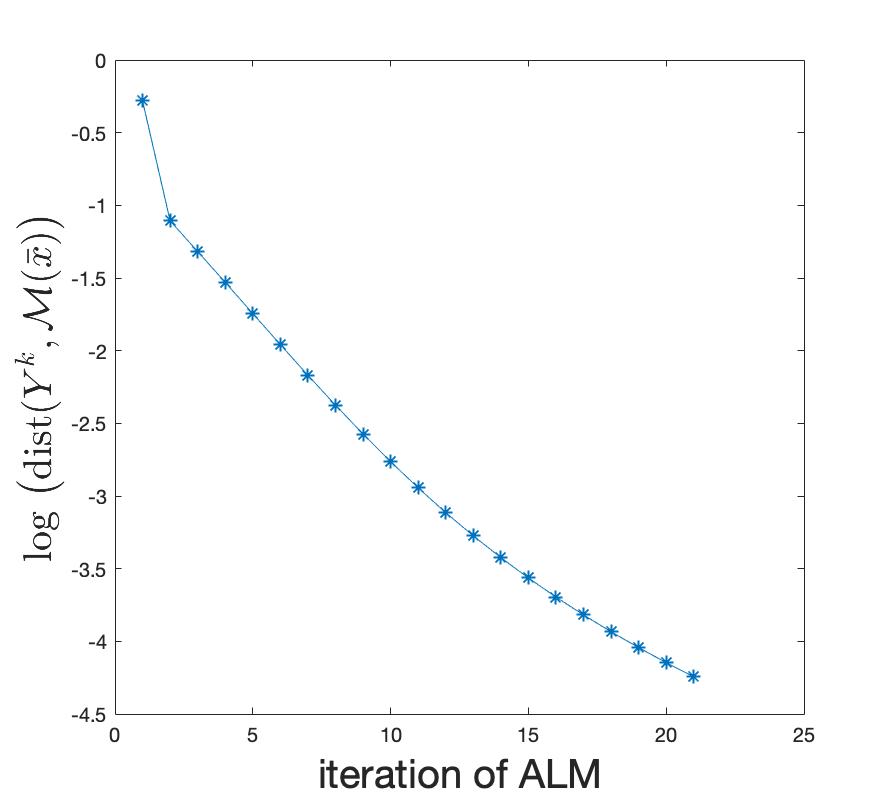} & \includegraphics[scale=0.16]{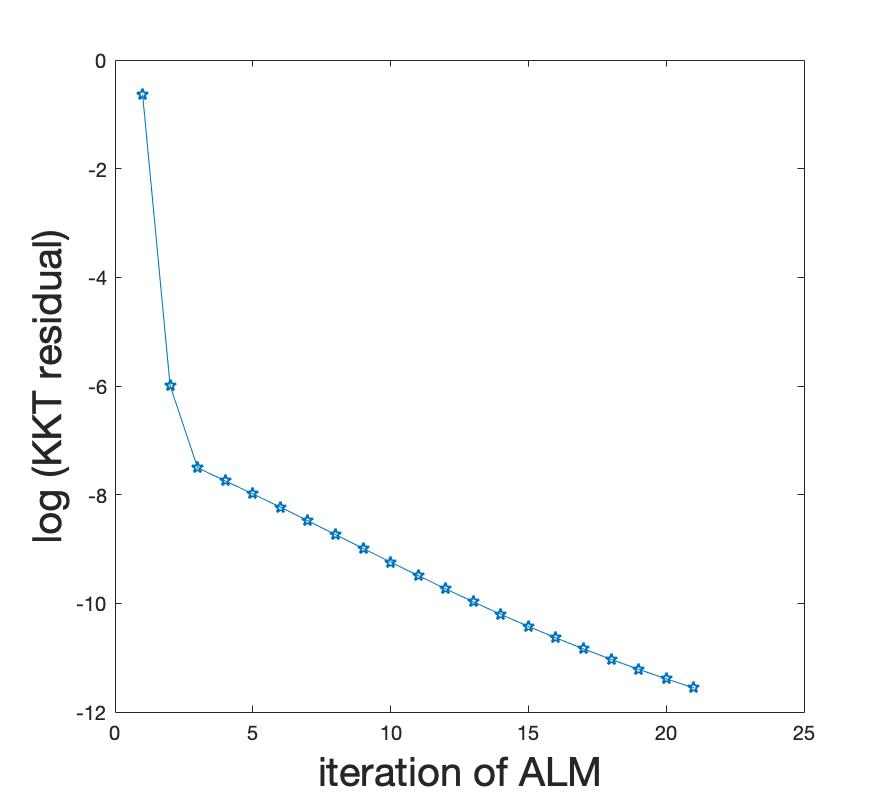}\\
		(a) & (b)\\
	\end{tabular}
	\caption{Semi-smooth Newton-CG based ALM  for problem \eqref{3times3} with $n=1000$ and $q=1500$. }\label{fig4}
	\vspace{-0.5em}
\end{figure}
It can be seen from Figure \ref{fig4} (a) that although the problem is nonconvex and the multiplier set is not a singleton, ALM also possesses a linear rate of convergence of ${\rm dist}((Y,Z),{\cal M}(\overline{X}))$ for the validity of strong variational sufficiency. Meanwhile, the right one shows that the KKT residual also converges to $0$ as the algorithm proceeds. Moreover, by strong variational sufficiency, we know from Theorem \ref{thm:varscequissosc} that the Hessian matrix $V_j$ in Algorithm \ref{algo2} Step 2 is always positive definite when $k$ is sufficiently large. Indeed, the minimum eigenvalue of the Hessian matrix $V_j$ is positive (e.g., $\lambda_{\min}(V_j)\approx 1.01$ for the case $n=3$ and $q=2$; $\lambda_{\min}(V_j)\approx 1.26$ for the case $n=100$ and $q=200$).  It is also worth to note that the distance from $(Y^k,Z^k)$ to $(\overline{Y},\overline{Z})$ does not converge to 0, which meets our theory as the limit point $(\widehat{Y},\widehat{Z})$ in Theorem \ref{thm:almrate} may be different from the reference point $(\overline{Y},\overline{Z})$.

	It is well-known that the condition number of $V_j$ in Algorithm \ref{algo2} Step 2 is proportional to $\rho^k$. In fact, the estimation of the condition number of $V_j$ can be obtained in a similar manner of  \cite[Lemma 10]{SSZhang08} as the smallest eigenvalue lies in a constant interval while the largest one lies in an interval which is propositional to $\rho^k$. Thus, in practice, we usually set an upper bound for $\rho^k$, although  $\rho^k$ can be infinity in theory.

It can be seen from Theorem \ref{thm:almrate} that the convergence analysis of Algorithm \ref{algo1} requires a good starting point of the multiplier while no restriction to the initial primal variable $x^0$ is needed. To verify the convergence theory (Theorem \ref{thm:almrate}), in the numerical experiments of \eqref{3times3} presented in Table \ref{tab}, we choose $X^0$ randomly. $(Y^0,Z^0)$ is chosen to be $(\overline{Y},\overline{Z})+\eta(P_1,P_2)$, where $P_1, P_2$ are symmetry matrices that are uniform randomly generated. Typically, $\eta$ is chosen to be small, e.g., 0.1.

Next, we will discuss how to generate the initial point more practically. Condition \ref{cond:closedc} indicates that $(Y^0,Z^0)$ should  be sufficiently close to $(\overline{Y},\overline{Z})$, at which strong SOSC is satisfied. How to find $(Y^0,Z^0)$ satisfying  Condition \ref{cond:closedc}  for nonconvex optimization problem is very challenging. A natural idea is to apply first order methods as a warm start to find a satisfactory initial point. 
Moreover, we know from Theorem \ref{thm:varscequissosc} that a good starting point may be the one at which the positive definiteness of generalized Hessian is satisfied. In practice, when the generalized Hessian is not positive definite, we may apply APG instead of the semismooth Newton to solve the ALM subproblem.

\section{Conclusion}
In this paper, we derive the equivalence between the strong variational sufficiency and the strong SOSC for NLSDP and NLSOC without requiring the uniqueness of multiplier or any other constraint qualification.
By using the equivalence result, the local convergence property of ALM for NLSDP can be established under merely strong SOSC instead of plus it with any constraint qualification. As a direct application, we are able to show that the positive definiteness of the generalized Hessian of augmented Lagrangian function, which is critical in the use of semi-smooth Newton method for NLSDP, is satisfied under strong SOSC. However, there are still many issues that deserve to be explored further. For example, we still do not have a satisfactory characterization for variational sufficiency. Moreover, for the implementation of ALM for solving nonconvex cases, it is still very challenging for finding a good starting point, e.g., the one close enough to a KKT pair of a local optimal solution which satisfies the strong SOSC condition.

\bigskip
\noindent{\bf Acknowledgements}
The authors would like to thank the two anonymous referees and the associate editor for their many valuable comments and constructive suggestions, which have substantially helped improve the quality and the presentation of this paper.

\end{document}